\documentclass{amsart}

\usepackage{a4wide,amsmath,amssymb}
\usepackage[toc,page]{appendix}
\usepackage{bbm}
\usepackage{url}
\usepackage{xcolor}

\newtheorem{thm}{Theorem}[section]

\newtheorem{prop}[thm]{Proposition}
\newtheorem{cor}[thm]{Corollary}
\theoremstyle{remark}
\newtheorem{rem}[thm]{Remark}

\theoremstyle{definition}

\newtheoremstyle{Claim}{}{}{\itshape}{}{\itshape\bfseries}{:}{ }{#1}
\theoremstyle{Claim}

\newcommand{\T}{{\mathbb{T}^d}}
\newcommand{\Z}{{\mathbb{Z}}}

\newcommand{\R}{\mathbb{R}}

\theoremstyle{plain}

\def\sideremark#1{\ifvmode\leavevmode\fi\vadjust{
\vbox to0pt{\hbox to 0pt{\hskip\hsize\hskip1em
\vbox{\hsize3cm\tiny\raggedright\pretolerance10000
\noindent #1\hfill}\hss}\vbox to8pt{\vfil}\vss}}}

\begin{document}

\title[]{On the optimal $L^q$-regularity for viscous Hamilton-Jacobi equations with subquadratic growth in the gradient}



\author{Alessandro Goffi}
\address{Dipartimento di Matematica ``Tullio Levi-Civita'', Universit\`a degli Studi di Padova, 
via Trieste 63, 35121 Padova (Italy)}
\curraddr{}
\email{alessandro.goffi@unipd.it}
\thanks{}

 \thanks{
 The author is member of the Gruppo Nazionale per l'Analisi Matematica, la Probabilit\`a e le loro Applicazioni (GNAMPA) of the Istituto Nazionale di Alta Matematica (INdAM). The author wishes to thank Prof. T. Leonori for fruitful discussions and suggesting many references on the subject of this paper. The author is indebted to the anonymous referee for the careful review which meant a significant improvement of the first version of the manuscript.
 }

\date{\today}

\begin{abstract}
This paper studies a maximal $L^q$-regularity property for nonlinear elliptic equations of second order with a zero-th order term and gradient nonlinearities having superlinear and subquadratic growth, complemented with Dirichlet boundary conditions. The approach is based on the combination of linear elliptic regularity theory and interpolation inequalities, so that the analysis of the maximal regularity estimates boils down to determine lower order integral bounds. The latter are achieved via a $L^p$ duality method, which exploits the regularity properties of solutions to stationary Fokker-Planck equations. For the latter problems, we discuss both global and local estimates. Our main novelties for the regularity properties of this class of nonlinear elliptic boundary-value problems are the analysis of the end-point summability threshold $q=d(\gamma-1)/\gamma$, $d$ being the dimension of the ambient space and $\gamma>1$ the growth of the first-order term in the gradient variable, along with the treatment of the full integrability range $q>1$.  
\end{abstract}

\maketitle

\tableofcontents
\section{Introduction}
In this note we establish maximal regularity properties in Lebesgue spaces for a large class of second order nonlinear elliptic equations, whose main model is the viscous Hamilton-Jacobi equation, equipped with Dirichlet boundary conditions of the form
\begin{equation}\label{hj}
\begin{cases}
- \Delta u(x)  +\lambda u(x)+H(x, Du(x)) = f(x) & \text{in $\Omega$}, \\
u(x) = 0 & \text{on $\partial\Omega$,}
\end{cases}
\end{equation}
with an unbounded right-hand side $f\in L^q(\Omega)$, $\Omega$ being a $C^2$ bounded domain of the $d$-dimensional Euclidean space $\R^d$, $\lambda\in\R$, $H=H(x,p)$ a nonlinearity with superlinear growth in the second entry. By maximal regularity in Lebesgue spaces we mean that an a priori information on the source term $f\in L^q$ implies bounds on the individual terms $D^2u,H(Du)$ on the left-hand side of the equation, see \cite{ChenWu,GT} for related properties for linear equations. Here, $H\in C(\Omega\times\R^d)$ is convex in the second variable and has superlinear growth, i.e. it satisfies for $\gamma>1$ the following assumption
\begin{equation}\label{H}\tag{H}
C_H^{-1}|p|^{\gamma}-C_H\leq H(x,p)\leq C_H(|p|^{\gamma}+1)
\end{equation}
for every $x\in \Omega$, $p\in\R^d$. 
We also assume that 
\begin{equation}\label{zeroth}
\lambda\geq 0\ .
\end{equation}
It is well-known that such an assumption is crucial to have at least existence of solutions in suitable Sobolev spaces, cf \cite{NotePorr}.\\
Second order estimates in $L^q$ in the linear and sublinear range of growth of the first-order term $0<\gamma\leq1$ are classical matter that can be inferred through the standard $W^{2,q}$ regularity theory: indeed, it is enough to apply the arguments in \cite[Section 3.5]{ChenWu} for linear problems, as already done for boundary-value problems of more general nonlinear equations with at most linear growth, see e.g. \cite[Proposition 2.6]{KS} (note that restrictions on $q$ are needed only for the fully nonlinear setting).\\
Instead, the problem of optimal gradient regularity in Lebesgue spaces for these classes of PDEs in the superlinear regime $\gamma>1$ has been proposed in a series of seminars by P.-L. Lions \cite{LionsSeminar,Napoli}, who conjectured its validity at least under the assumption that  
\begin{equation}\label{qdg}q>\frac{d(\gamma-1)}{\gamma}=:q_{d,\gamma},\end{equation}
 $\gamma>1$, $q\geq1$, $d$ being the dimension of the ambient space. Except for some specific cases discussed in his seminars, see e.g. \cite[pp.1522-1523]{CGell} or Section \ref{op} for more details and open problems, the conjecture was investigated in \cite{CGell} for equations posed on the flat torus (i.e. the case of the ergodic control of a diffusion) without zero-th order terms in the equation, under the coupled assumption
\[
q>q_{d,\gamma}\text{ and }q>2,
\]
through a delicate refinement of the integral Bernstein method, that in turn forced to add the additional restriction $q>2$. Nonetheless, the above coupled condition is always realized in a suitable range of the parameter $\gamma$, e.g. when $\gamma>\frac{d}{d-2}$.\\
Notice that the condition \eqref{qdg} is natural if one regards the problem as a nonlinear Poisson equation: indeed, if $u$ solves \eqref{hj} with $\lambda=0$, then it is a solution to
\[
-\Delta u=-H(x,Du)+f(x)\ ,
\]
so that classical maximal $L^q$-elliptic regularity properties, \eqref{H} and Sobolev embeddings lead to the estimates
\[
\|Du\|_{L^{\frac{dq}{d-q}}(\Omega)}\lesssim \|u\|_{W^{2,q}(\Omega)}\lesssim \||Du|\|_{L^{q\gamma}(\Omega)}^\gamma+\|f\|_{L^q(\Omega)}.
\]
Therefore, one can expect maximal regularity properties for the nonlinear problem whenever the integrability of the gradient on the right-hand side is less than the one on the left, i.e. whenever $q^*=\frac{dq}{d-q}\geq q\gamma$, which implies $q\geq\frac{d(\gamma-1)}{\gamma}$.\\
 In this note we continue the analysis on this regularity problem for different boundary conditions, as suggested by P.-L. Lions in his seminars, but for the sole range $\gamma< 2$. Indeed, the superquadratic case $\gamma>2$ and the full range $q\geq\frac{d(\gamma-1)}{\gamma}$ require a different treatment and new estimates at the level of H\"older spaces that have been recently addressed in \cite{CV,c22par} through blow-up methods, so we do not pursue on this direction. Nonetheless, the duality method of this manuscript would even provide results in the regime $\gamma>2$, but at the expenses of strengthening the summability of $f$, as in \cite{CGpar}.\\
 We first show that it is possible to have a control on $D^2u,|Du|^\gamma\in L^q$ in terms of $f\in L^q$ removing the above restriction $q>2$ in the whole subquadratic regime $\gamma<2$ proposing a different proof with respect to \cite{CGell}. Here, the approach is inspired by that employed for the parabolic problem in \cite{CGpar}, and will be based on a perturbation argument that combines $L^q$-regularity results for linear problems, cf Appendix \ref{app}, and interpolation inequalities \cite{Nirenberg}, handling Dirichlet boundary conditions. As a consequence, the obtainment of maximal $L^q$-regularity estimates through interpolation methods boils down to establish a lower order bound for solutions to \eqref{hj} in Lebesgue spaces. We deduce such a level of regularity via duality arguments, following \cite{cg20,CGpar}. The study of regularity properties for linear problems (even at the level of maximal regularity) shifting the attention to their adjoint equations is a classical idea, see e.g. \cite{EvansTAMS,Lin} and the references therein, and it has been recently and systematically explored in the nonlinear setting of Hamilton-Jacobi equations starting with the work by L.C. Evans \cite{Evansacc}, and later in the context of Mean Field Games, see e.g. \cite{Gomesbook,CGpar} for a thorough bibliography. Due to the presence of a nonlinear term, this approach is tied up with the smoothness properties, at the level of Sobolev spaces, of a (dual) stationary Fokker-Planck equation with a ``singular'' forcing term of the form
\begin{equation}\label{adjintro}
\begin{cases}
-\Delta\rho+\lambda\rho+\mathrm{div}(b(x)\rho)=\delta_{x_0}& \text{in $\Omega$} \\
\rho = 0 & \text{on $\partial\Omega$,}
\end{cases}
\end{equation}
 where $\delta_{x_0}$ stands for the Dirac measure at some $x_0\in\Omega$, when the drift $b$ satisfies some a priori integrability conditions against the solution $\rho$ itself. \\
As a second aim, we discuss the maximal regularity property even at the critical level of summability $q=\frac{d(\gamma-1)}{\gamma}$, thus providing new advances on this regularity problem with respect to \cite{CGell}. In this case, as previously observed in \cite{GMP14}, we enlighten the role of the zero-th order term in the equation with respect to the dependence on the right-hand side datum $f\in L^q$. Even though the existence of solutions require in general $\lambda>0$, see \cite{NotePorr}, we emphasize that regularity or a priori estimates are expected to hold up to $\lambda=0$, as in \cite{GMP14}, the difference among the cases $\lambda>0$ and $\lambda=0$ being in the dependence of the maximal regularity estimate.\\
 
 The first main maximal $L^q$-regularity property in the \textit{subcritical regime} $q>\frac{d(\gamma-1)}{\gamma}$ will be addressed in Theorem \ref{main} through lower-order bounds on solutions to \eqref{hj} in Lebesgue spaces, see Corollary \ref{finalint}. We emphasize that this seems, to the author's knowledge, a slightly novel viewpoint of the analysis of integral estimates in the framework of maximal regularity for these nonlinear elliptic equations with coercive gradient terms and Dirichlet boundary conditions. However, we borrow several ideas already appeared in the context of parabolic equations from \cite{cg20,CGpar}. It is worth pointing out that a similar $L^p$ duality argument was developed in a rather different context by M. Pierre \cite[Section 3.2]{Pierre} to prove integral estimates for solutions to time-dependent reaction-diffusion systems, while some duality methods in the context of elliptic boundary value problems were previously used in various papers, see e.g. \cite{BensFrehse,BensoussanBook,BocBuc,BocJDE,BOP}, the more recent \cite{BO} and the references therein, but at different summability scales and for different notions of solutions. We note that the integral estimates we obtain in Corollary \ref{finalint} for \eqref{hj} in the case $q>\frac{d(\gamma-1)}{\gamma}$ are not new when $\gamma<2$ and they can be found in \cite{GMP14} for more general equations with diffusion operators in divergence-form modeled over the $p$-Laplacian. Nevertheless, they were obtained using completely different methods, yet based on test function arguments, under the restriction $\gamma<p$ (when the leading operator is the $p$-Laplacian). Hence, our method of proof seems new with respect to the current literature of boundary-value problems. Instead, some integral bounds for distributional subsolutions when $\gamma>2$ can be found in \cite{DAP}. As far as the $L^\infty$ bounds are concerned, our methods unify the treatment in \cite{GMP14,DAP} and allow to handle both the sub- and superquadratic regimes of the nonlinearity.\\
 Our second main result is a treatment of the maximal regularity problem at the end-point summability threshold $q=d(\gamma-1)/\gamma$, assuming further integrability assumptions on the right-hand side $f\in L^q$. This kind of dependence was already pointed out for estimates at lower integrability scales, see \cite{GMP14}. Indeed, we are able to show that when $\lambda>0$ the maximal regularity estimate depends on the equi-integrability properties of the datum $f$ in $L^q$. Instead, when $\lambda=0$, the maximal regularity estimates depend also on $u\in L^q$ other than $f$, and to obtain a full estimate depending solely on the integrability of the source of the equation one needs to impose further smallness conditions on $f$, see Remark \ref{remfull} and also the recent work \cite{CGL}. Our proof is inspired by a stability argument recently proposed in the parabolic setting in \cite[Theorem 1.3]{CGpar} and, to our knowledge, it cannot be inferred from the known results in the literature, e.g. from \cite{CGell,GMP14}. \\
Still, we remark that the method we propose here to get global maximal regularity does not require any geometric assumption on the domain. Indeed, the use of the Bernstein method typically requires the convexity of the state space \cite{GP}, at least for Neumann boundary value problems.\\
 
As already mentioned, the integral estimates we obtain by duality are achieved through Sobolev regularity estimates for stationary Fokker-Planck equations having drifts satisfying suitable summability conditions against the solution itself. An additional distinctive feature of our method is to provide new global bounds for the Dirichlet problem \eqref{adjintro} as well as local estimates for solutions to linear equations with lower-order coefficients in divergence form, when $|b|\in L^k(\rho\,dx)$, for some $k>1$, that might be of independent interest. This goal is achieved by adapting the approach of earlier works on the subject \cite{MPRell,cg20,CGpar}.\\
Then, the integral estimates for \eqref{hj} follow by observing that an integrability information on the quantity $D_pH(x,Du)$ along the solution controls the regularity in Lebesgue scales of the solution $u$ itself, since in turn it controls the (Sobolev) regularity of the solution of the dual problem \eqref{adjintro} with $b=-D_pH(x,Du)$, as in \cite{cg20,CGpar}. This crucial step is achieved through an integral representation formula, see the proofs of Propositions \ref{upiu} and \ref{umeno} for further details.\\
 We remark that our variation on the Evans' scheme seems to be the first application to stationary Hamilton-Jacobi equations with $L^p$ coefficients, and also to different boundary conditions than the periodic setting, except the work \cite{EvansSmart} on differentiability properties of solutions to boundary-value problems of equations modeled over the $\infty$-Laplacian. We refer to \cite{Gomesbook,cg20,CGpar,TranBook} and the references therein for further developments and earlier results through this nonlinear duality method.\\
 
It is worth pointing out that some maximal regularity results have been already obtained in the literature. Other than the aforementioned paper \cite{GMP14}, when $H$ is slowly increasing in the gradient variable (precisely when $\gamma\leq\frac{d+2}{2}$, which prevents from the use of energy formulations), the paper \cite{MercaldoCPAA} addressed gradient regularity properties when the source $f$ belongs to finer Lorentz classes without zero-th order terms, see also the references therein and Remark \ref{sulorentz}. Still, a quite general regularity theory on spaces of maximal regularity for such equations, mostly for systems of Hamilton-Jacobi-Bellman equations with bounded data and first-order terms having at most quadratic growth, can be found in \cite{BensoussanBook}. Some other results on second order Sobolev spaces via the Amann-Crandall approach \cite{AC} can be found in the monograph \cite{MaSoBook}, and are based on Aleksandrov-Bakel'man-Pucci estimates, which in turn require $f\in L^d$.\\ 
We conclude by saying that our underlying motive for this analysis relies on the application of the maximal regularity estimates to the existence problem of classical solutions to the systems of PDEs arising in the theory of Mean Field Games introduced by J.-M. Lasry-P.-L. Lions \cite{LL07}, when the coupling term of the Hamilton-Jacobi equation has power-like growth, cf \cite{CCPDE,CC} and \cite[Theorems 1.4 and 1.5]{CGpar}. Some recent results in this direction through maximal regularity for the so-called defocusing systems \cite{CCPDE} posed on convex domains of the Euclidean space and Neumann boundary conditions have been discussed in \cite{GP}.\\
Finally, an extension of this method to the parabolic case would allow to treat even the following Cauchy-Dirichlet problem when $\gamma<2$
\[
\begin{cases}
\partial_tu-\Delta u+H(x,Du)=f(x,t)&\text{ in }\Omega\times(0,T),\\
u(x,0)=u_0(x)&\text{ in }\Omega,\\
u(x,t)=0&\text{ on }\partial\Omega\times(0,T),
\end{cases}
\]
with suitable changes in the integrability exponent $q$ of $f\in L^q(\Omega\times(0,T))$, leading to the counterpart of the results in \cite{CGpar} for the case of bounded domains with homogeneous Dirichlet boundary conditions.\\
 We finally mention that our technique does not apply to problems driven by the $p$-Laplacian, as in \cite{GMP14}, though some recent advances on the problem of maximal regularity in Lebesgue spaces for the quasilinear counterpart of \eqref{hj} have been recently obtained in \cite{CGL} using different integral Bernstein methods. 

\textit{Plan of the paper}. Section \ref{mainres} is devoted to present the statements of the main results of the manuscript along with some related remarks. In Section \ref{estfp} we establish the Sobolev regularity of solutions for the adjoint Fokker-Planck equation, analyzing both global and interior estimates, while Section \ref{estu} studies the integral estimates for \eqref{hj} by duality methods. Section \ref{max} discusses the proofs of the main results. Section \ref{quad} and Appendix \ref{app} conclude the paper with some remarks on the case of the quadratic growth and an auxiliary Calder\'on-Zygmund result for linear problems.
\section{Main results}\label{mainres}
From now on, $\Omega\subset\R^d$ will be a $C^{2}$ bounded domain. This is only one of the possible regularity conditions one can assume on the domain to ensure the simultaneous validity of the Sobolev's inequality as well as Calder\'on-Zygmund estimates for the Dirichlet problem. The latter will be recalled in Theorem \ref{cz}.\\
Our main results are the following: the first one deals with the subcritical regime $q>\frac{d(\gamma-1)}{\gamma}$, while the second one focuses on the critical threshold of summability $q=\frac{d(\gamma-1)}{\gamma}$. We will only detail the case $\lambda>0$, since
\[
-\Delta u+H(x,Du)=f(x)\implies -\Delta u+\lambda u+H(x,Du)=g(x),\ g(x)=f(x)+\lambda u.
\]
Therefore, in the case $\lambda=0$ bounds will depend also on the integrability properties of $u\in L^q$ and the result will then follow applying the integral estimates from \cite{GMP14,FM} in the case $\lambda=0$.
\begin{thm}\label{main}
Let $\Omega\subset\R^d$ be a $C^2$ bounded domain, $d\geq 1$, and $u\in W^{2,q}(\Omega)\cap W^{1,q}_0(\Omega)$, $q>\frac{d(\gamma-1)}{\gamma}$, $1+\frac2d\leq\gamma\leq 2$, be a strong solution to \eqref{hj} and assume \eqref{H}-\eqref{zeroth}. Suppose that
\[
f\in L^q(\Omega)\ ,q>\frac{d(\gamma-1)}{\gamma}.
\]
Then there exists a constant $M_1$ depending on $d,q,\gamma,C_H,|\Omega|,\lambda,\|f\|_{L^q(\Omega)}$ (and on $\|u\|_{L^q(\Omega)}$ when $\lambda=0$) such that
\[
\|u\|_{W^{2,q}(\Omega)}+\||Du|^\gamma\|_{L^q(\Omega)}\leq M_1
\]
\end{thm}
Our approach closely follows \cite{CGpar} and can be described as follows: first, considering  \eqref{hj} as a nonlinear Poisson equation, by linear maximal regularity and \eqref{H}, any strong solution to \eqref{hj} satisfies
\[
\|D^2 u\|_{L^q}\lesssim \||Du|\|_{L^{q\gamma}}^\gamma+\|f\|_{L^q}\ .
\]
Using standard Gagliardo-Nirenberg interpolation inequalities one gets
\[
\|Du\|_{L^{q\gamma}}^\gamma\lesssim \|D^2u\|_{L^q}^{\theta\gamma}\|u\|_{L^s}^{(1-\theta)\gamma}
\]
with $\theta\gamma<1$ when $s$ is suitably chosen. This reduces the maximal regularity estimate, through the application of the weighted Young's inequality, to a lower order estimate in Lebesgue spaces, see Corollary \ref{finalint}. The latter is accomplished by duality methods through the study of maximal regularity properties of stationary Fokker-Planck equations, cf Corollary \ref{corfpreg}.\\
The second main result deals with the critical value of summability: in this case, the above interpolation argument leads to the condition $\theta\gamma=1$, so that the Young's inequality does not allow to conclude the statement. It is based on a careful analysis of stability estimates in Lebesgue spaces, as stated in Proposition \ref{stabLpest} below.
\begin{thm}\label{main2}
Let $\Omega\subset\R^d$ be as in Theorem \ref{main}, $d\geq 1$, and $u\in W^{2,q}(\Omega)\cap W^{1,q}_0(\Omega)$, $1+\frac2d\leq\gamma<2$, be a solution to \eqref{hj} satisfying \eqref{H}-\eqref{zeroth} and
\[
f\in L^q(\Omega), q=\frac{d(\gamma-1)}{\gamma}.
\] 
Then there exists a constant $M_2$ depending on $q,\gamma,C_H,|\Omega|$ and $f$ such that
\[
\|u\|_{W^{2,q}(\Omega)}+\||Du|^\gamma\|_{L^q(\Omega)}\leq M_2.
\]
In particular, 
\begin{itemize}
\item[(i)] when $\lambda>0$ the constant $M_2$ depends also on $\lambda$ and does not depend only on $\|f\|_{L^q(\Omega)}$, but remains bounded when $f$ varies in a set which is bounded and equi-integrable in $L^q(\Omega)$;
\item[(ii)] if $\lambda=0$ the constant $M_2$ depends also on the equi-integrability properties of $u\in L^q(\Omega)$;
\end{itemize}
\end{thm}
Some remarks on the above results are now in order. In the sequel, $\gamma'$ will denote the H\"older conjugate exponent of $\gamma$, i.e. $\gamma'=\frac{\gamma}{\gamma-1}$.
\begin{rem}\label{remfull}
The dependence on $u\in L^q$ in the statement (ii) of Theorem \ref{main2} can be removed exploiting the estimates obtained in e.g. Theorem 1.1-(ii) of \cite{GMP14} applied with $p=2$, see also \cite{FM}, under a smallness condition on $f$ (cf. condition (1.5) in \cite{GMP14}). This agrees with the statement of Theorem 3.1-(i) of \cite{CGL}. Similarly, when $\lambda=0$, the estimate can be made independent of the norm $\|u\|_{L^q}$ even in Theorem \ref{main} using e.g. Theorem 3.8 in \cite{GMP14}.
\end{rem}
\begin{rem}
In Theorems \ref{main} and \ref{main2} we will only focus on the case $d>2$. Indeed, a careful inspection of the proofs shows that the lower dimensional cases $d=1$ and $d=2$ are simpler to address due to the embedding of $W^{1,2}$ into any $L^p$ space for finite $p$. 
\end{rem}
\begin{rem}\label{nondiv}
The results of this manuscript can be extended to more general second order diffusions of the form $-\mathrm{Tr}(A(x)D^2u)$. We emphasize that some control on the derivatives of $A$ is needed to study the regularity of the dual Fokker-Planck equation via the auxiliary problem \eqref{adjintro}, while Theorem \ref{cz} merely requires the leading coefficients to be continuous in the whole domain, cf \cite[Lemma 9.17]{GT}.
Partial results with Sobolev regularity assumptions on the diffusion matrix $A$ already appeared in \cite{BardiPerthame} through the integral Bernstein method to prove Sobolev estimates for problems with first order terms having natural growth, under further integrability conditions on the source term. It is still unclear which are the minimal regularity requirements for the validity of the maximal $L^q$-regularity for viscous Hamilton-Jacobi equations driven by operators in non-divergence form. However, being our methods of variational nature, extensions to more general diffusions of the form $-\mathrm{Tr}(A(x)D^2u)$ would be possible provided that the coefficients are smooth enough, e.g. Lipschitz continuous, since $a_{ij}\partial_{ij}u=\partial_i(a_{ij}\partial_ju)-\partial_ia_{ij}\partial_ju$, which in turn results in the presence of an additional transport term with a bounded coefficient. 
\end{rem}
\begin{rem}
The previous main results can be extended to non-homogeneous boundary conditions via the corresponding results for the linear problem, see e.g. \cite[Theorem 9.15]{GT} or \cite[Section 5]{ChenWu}. Moreover, the techniques to prove Theorems \ref{main} and \ref{main2} apply identically to problems with zero-th order terms posed on the flat torus $\T\equiv \R^d/\Z^d$ typical of the ergodic control setting, i.e. with periodic data, or with other boundary conditions (e.g. of Neumann and mixed type, see \cite[Section 7.1]{GMP14}), leading to new results even in these contexts. Still, we emphasize that the same strategy presented here could also be extended with appropriate modifications to study parabolic Cauchy-Dirichlet problems in the subquadratic regime through the corresponding result for the linear problem. In this case, the parabolic dimension $d+2$ would replace the dimension $d$ of the state space in the integrability conditions, see \cite{CGpar} for additional details.
\end{rem}

\section{Remarks on the growth of the first-order term and the integrability of $f$}
At this stage, it is worth recalling that the growth regime of the parameter $\gamma$ as well as the integrability $q$ play a crucial role to investigate the regularity properties of our nonlinear Dirichlet problem. 
 First, since the integrability of $f$ must be also larger than one, we will need to assume 
 \[
 q\geq\max\left\{\frac{d(\gamma-1)}{\gamma},1\right\}.
 \]
As a consequence, as far as the growth of the first order term is concerned, below the natural growth regime $\gamma=2$ one usually identifies the other two critical exponents $\gamma=\frac{d}{d-1}$ and $\gamma=1+\frac{2}{d}$. The former appears from the exponent $q=\frac{d(\gamma-1)}{\gamma}$, which satisfies $q\geq1$ whenever $\gamma\geq \frac{d}{d-1}$. The latter can be identified arguing by linearization: indeed, \eqref{hj} can be written as
 \[
 -\Delta u+|Du|^{\gamma-2}Du\cdot Du=f(x)\in L^q
 \]
 so that setting $B(x)=|Du|^{\gamma-2}Du$, one could achieve maximal $L^q$-regularity when $B(x)\in L^r(\Omega)$, $r\geq d$ (cf \cite{GT}). Starting with energy solutions $u\in W^{1,2}_0$ we would conclude $B\in L^{\frac{2}{\gamma-1}}$. Therefore, we need
 \[
 \frac{2}{\gamma-1}\geq d\implies \gamma\leq 1+\frac{2}{d}.
 \]
 Such a linearization argument has been already used to study uniqueness of solutions \cite{BarlesPorretta} and integral estimates coming from maximal regularity in \cite{GMP14}. Therefore, it is natural to distinguish among the regimes
 \[
 0<\gamma\leq 1\ (\text{ i.e. sublinear/linear growth});
 \] 
 \[
 1<\gamma<\frac{d}{d-1} \left(\implies q=1\right);
 \]
\[
\frac{d}{d-1}<\gamma<1+\frac{2}{d}\ \left(\iff 1<q<\frac{2d}{d+2}=(2^*)'\right);
\]
\[
1+\frac{2}{d}\leq\gamma< 2\ \left(\iff (2^*)'=\frac{2d}{d+2}\leq q<\frac{d}{2}\right),
 \]
$2^*=\frac{2d}{d-2}$ being the Sobolev exponent. In particular, the last regime allows to use the energy formulation of the problem. Indeed, considering strong solutions in the standard sense of \cite[Chapter 9]{GT} for \eqref{hj}, i.e. such that $u\in W^{2,q}(\Omega)\cap W^{1,q}_0(\Omega)$, we would have the embedding
\[
W^{2,q}\hookrightarrow W^{1,2}\ ,
\]
which occurs when $q\geq \frac{2d}{d+2}$. Then, $\frac{d}{\gamma'}\geq \frac{2d}{d+2}$ precisely when $\gamma\geq 1+\frac{2}{d}$. When $\gamma$ is slowly increasing, precisely below $\gamma= 1+\frac{2}{d}$, one falls outside the energy range, as the datum $f$ may not belong to $W^{-1,2}(\Omega)$, and needs different notions of solutions, such as renormalized solutions or those obtained as limit of approximations, see \cite[Section 4]{GMP14} and \cite{MercaldoCPAA}.\\
Despite all the main results Theorems \ref{main} and \ref{main2} could have been treated in the full range $\gamma\in(1,2)$, we will only detail the case within the energy range $1+\frac{2}{d}\leq \gamma<2$, since the former have been already discussed in \cite{Napoli} and in Theorems 3.6-3.8 of \cite{MercaldoCPAA}, even for problems with data in Lorentz spaces. However, one can address the maximal $L^q$-regularity following the path of Section \ref{max} through the integral estimates proved in e.g. \cite{GMP14} when working in the regime $1<\gamma<1+\frac{2}{d}$, so we omit this analysis. Still, we do not detail the case $\gamma>2$, as it requires estimates at the level of H\"older spaces, see Section \ref{quad}.

\section{The case $1+\frac{2}{d}\leq \gamma<2$}
\subsection{Preliminary results for stationary Fokker-Planck equations}\label{estfp}
\subsubsection{Global Sobolev regularity for the Dirichlet problem}
In this section we give some global regularity results for the (dual) Dirichlet problem
\begin{equation}\label{fp}
\begin{cases}
-\Delta \rho(x)+\lambda \rho(x)+\mathrm{div}(b(x)\rho(x))=\psi(x)& \text{in $\Omega$,} \\
\rho(x) = 0 & \text{on $\partial\Omega$.}
\end{cases}
\end{equation}
Here, $\psi$ should be thought as a $L^p$ approximation of a Dirac delta, for $p$ to be specified later, cf \cite{Evansacc}. We denote, as usual, by $W^{k,p}_0(\Omega)$ the closure of $C_0^k(\Omega)$ in $W^{k,p}(\Omega)$, while we consider weak solutions to \eqref{fp} belonging to $W_0^{1,2}(\Omega)$ in the sense that the following identity holds
\[
\int_\Omega D\rho\cdot D\varphi\,dx+\lambda\int_\Omega \rho\varphi\,dx-\int_\Omega b\rho\cdot D\varphi\,dx=\int_{\Omega} \psi\varphi\,dx\ ,\forall \varphi\in W_0^{1,2}(\Omega)\ .
\]
We start with a well-known well-posedness result for the Dirichlet problem.
\begin{prop}\label{well}
Let $\Omega\subset\R^d$, $d>2$, be a bounded domain, $b\in L^p(\Omega)$, $p\geq d$, and $\psi\in L^\infty(\Omega)$. Then, there exists a unique weak solution $\rho\in W^{1,2}_0(\Omega)$ of the Dirichlet problem \eqref{fp}. When $b\in L^p(\Omega)$, $p> d$, $\rho\in L^\infty(\Omega)$, while when $p=d$ one has $\rho\in L^r(\Omega)$ for any finite $r>1$. Moreover, if $\psi\geq0$ in $\Omega$, then $\rho\geq0$ in $\Omega$. Finally, if $\psi\in L^1(\Omega)$ we have $\|\rho\|_{L^1(\Omega)}\leq \frac{\|\psi\|_{L^1(\Omega)}}{\lambda}$.
\end{prop}
\begin{proof}
The existence and uniqueness follow from \cite[Proposition 2.1.4]{BKRS} combined with \cite[Theorem 2.1.8]{BKRS} (applied with $\gamma=0$ and $\beta\equiv0$, which requires $\lambda\geq0$ using the notation of this paper, cf \cite[Remark 2.1.10]{BKRS}), while the positivity is a consequence of \cite[Theorem 2.2.1-(i) and (iii)]{BKRS}. The $L^\infty$ estimates when $b\in L^p$, $p>d$, are standard and can be obtained through the estimate of $\log(1+|u|)\in W^{1,2}$, cf \cite[Theorem 5.6]{BocUMI} or \cite[Theorem 2.1]{BOPpar}, while the estimates in $L^r(\Omega)$ when $b\in L^d$ can be obtained arguing as in \cite[p. 414]{BOPpar}, see also \cite[Theorem 5.5]{BocUMI}.
The last assertion 
\begin{equation}\label{rhoL1}
\int_\Omega \rho\leq \frac{\int_\Omega \psi}{\lambda}
\end{equation}
was proved in e.g. \cite[Lemma 4.1]{BocJDE} (where it is enough to have $b\in L^2(\Omega)$). 
%

\end{proof}
We now prove the following regularity result for solutions to the adjoint problem \eqref{fp} with Dirichlet boundary conditions in terms of an integrability information of the drift term against the solution itself, cf \cite{cg20,CCPDE} for deeper results on the subject and the importance in connection with Mean Field Games. The proof is inspired by some ideas already appeared in \cite{MPRell} (cf \cite{BKR} for a different approach), see also \cite{cg20,CGpar} for related results for the parabolic problem along with the general reference \cite{BKRS}. Note that here the right-hand side term $\psi$ plays the same role of the terminal datum of the backward adjoint problem in the parabolic case analyzed in \cite{CGpar}. Our result extends with a different proof those in \cite{BKR} and provides, in addition, an explicit control on the size of the Sobolev and Lebesgue norms.
\begin{prop}\label{regfp1}
Let $1<\sigma'<d$. Let then $\psi\in L^{p'}(\Omega)$ and $|b|\in L^k(\Omega;\rho\,dx)$ with $k=1+\frac{d}{\sigma}$, with $k,\sigma,p$ satisfying the following conditions
\begin{itemize}
\item $p=\frac{d}{k-2}$ when $\sigma'>\frac{d}{d-1}$ with $2<k<1+\frac{d}{2}$;
\item $p'=1$ when $1<\sigma'<\frac{d}{d-1}$;
\item any $p'<\infty$ when $\sigma'=\frac{d}{d-1}$.
\end{itemize}
Then, every nonnegative weak solution to the Dirichlet problem \eqref{fp}  satisfies the estimate
\[
\|\rho\|_{L^{\frac{d}{d-k}}(\Omega)}+\|D\rho\|_{L^{\sigma'}(\Omega)}\leq C\left(\int_\Omega |b|^k\rho\,dx +\|\psi\|_{L^{p'}(\Omega)}+1\right)
\]
where $C$ depends only on $d,\sigma',|\Omega|$ when $\sigma'\geq\frac{d}{d-1}$, while it depends also on $\lambda$ when $\sigma'< \frac{d}{d-1}$.
\end{prop}
\begin{proof}
We first discuss the case $\sigma'>\frac{d}{d-1}$ using a variational argument. Let $\beta=\frac{k-2}{d-k}$ and use the test function $\varphi=\rho^\beta$. However, since $\beta\in(0,1)$, it follows that $D\varphi$ may not be in $L^2$, and hence one has to make the argument rigorous by taking $\varphi=(\delta+\rho)^\beta-\delta^\beta\in W_0^{1,2}$, where $\delta>0$, and then let $\delta\to0$. For the sake of presentation, we derive some formal estimates using $\varphi=\rho^\beta$ as a test function in the weak formulation of \eqref{fp}. We have
\[
\int_\Omega D\rho\cdot D(\rho^\beta)\,dx-\int_\Omega \rho b\cdot D(\rho^\beta)\,dx+\lambda\int_\Omega \rho^{\beta+1}\,dx=\int_\Omega \rho^\beta \psi\,dx\ ,
\]
which is equivalent to
\[
\int_\Omega |D\rho|^2\rho^{\beta-1}\,dx+\frac{\lambda}{\beta}\int_\Omega \rho^{\beta+1}\,dx
\leq\int_\Omega |b|\rho^\beta|D\rho|\,dx+\frac1\beta\int_\Omega \rho^\beta\psi\,dx\ .
\]
We then use the generalized Young's inequality on the first term of the right-hand side to get
\begin{multline}\label{mainfp}
\frac12\int_\Omega |D\rho|^2\rho^{\beta-1}\,dx+\frac12\int_\Omega |D\rho|^2\rho^{\beta-1}\,dx+\frac{\lambda}{\beta}\int_\Omega \rho^{\beta+1}\,dx\leq\frac14\int_\Omega |D\rho|^2\rho^{\beta-1}\,dx\\
+\int_\Omega\rho^{\beta+1}|b|^2\,dx  +\frac1\beta\int_\Omega \rho^\beta\psi\,dx\ .
\end{multline}
As for the left-hand side, we observe that by Sobolev's inequality we have
\[
\int_\Omega |D\rho|^2\rho^{\beta-1}\,dx=c_\beta\int_\Omega |D\rho^{\frac{\beta+1}{2}}|^2\\
\geq c_{\beta,d}\left(\int_\Omega \rho^{(\beta+1)\frac{d}{d-2}}\,dx\right)^{1-\frac2d}.
\]
Then, writing $\int_\Omega\rho^{\beta+1}|b|^2\,dx=\int_\Omega |b|^2\rho^\frac{2}{k}\rho^{\beta+\frac{k-2}{k}}\,dx$, we first apply H\"older's inequality with exponents $\left(\frac{k}{2},\frac{k}{k-2}\right)$ and $(p,p')$, $p=\frac{d}{k-2}$, respectively to the first and second term of the right-hand side of \eqref{mainfp}, and then the generalized Young's inequality with the pairs of conjugate exponents $\left(\frac{(d-2)k}{2(d-k)},\frac{(d-2)k}{d(k-2)}\right)$ and $\left(\frac{d-2}{d-k},\frac{d-2}{k-2}\right)$ respectively to get
\begin{multline*}
\int_\Omega \rho^{\beta+1}|b|^2\,dx+\frac1\beta\int_\Omega \rho^\beta\psi\,dx \\
\leq \left(\int_\Omega |b|^k\rho\,dx\right)^{\frac2k}\left(\int_\Omega \rho^{\beta\frac{k}{k-2}+1}\,dx\right)^{1-\frac2k}+\frac1\beta\|\psi\|_{L^{p'}(\Omega)}\left(\int_\Omega\rho^{\beta p}\right)^\frac1p\\
\leq \tilde c_{d,\beta}\left[\left(\int_\Omega |b|^k\rho\,dx\right)^\frac{d-2}{d-k}+\|\psi\|_{L^{p'}(\Omega)}^\frac{d-2}{d-k} \right] +\frac{c_{d,\beta}}{8}\left(\int_\Omega \rho^{\beta\frac{k}{k-2}+1}\,dx\right)^{1-\frac2d}+\frac{c_{d,\beta}}{8}\left(\int_\Omega\rho^{\beta p}\right)^{1-\frac2d}.
\end{multline*}
One immediately checks the validity of the following chain of identities
\[
\frac{d}{d-k}=(\beta+1)\frac{d}{d-2}=\beta\frac{k}{k-2}+1=\beta p\ .
\]
Then, for some positive constant $C$ depending solely on $d,\beta$ we get
\[
\left(\int_\Omega \rho^{\frac{d}{d-k}}\,dx\right)^\frac{d-2}{d}+\frac14\int_\Omega \rho^{\beta-1}|D\rho|^2\,dx\leq C\left(\int_\Omega |b|^k\rho\,dx+\|\psi\|_{L^{p'}(\Omega)}\right)^\frac{d-2}{d-k}
\]
giving the desired estimate on $\rho\in L^{\frac{d}{d-k}}(\Omega)$. Exploiting the fact that $\beta\in(0,1)$ we conclude by the H\"older's inequality (recalling that $\sigma'=\frac{d}{d-k+1}$ and $k<1+\frac{d}{2}$)
\[
\|D\rho\|_{L^{\sigma'}(\Omega)}=\|D\rho\|_{L^{\frac{d}{d-k+1}}(\Omega)}\leq \|\rho^{\frac{\beta-1}{2}}D\rho\|_{L^2(\Omega)}\|\rho^{\frac{1-\beta}{2}}\|_{L^{\frac{2d}{d-2k+2}}(\Omega)}
\]
and using the previous estimates we conclude the assertion.\\
The proof in the case $\sigma'<\frac{d}{d-1}$ is based on maximal regularity arguments, and it can be obtained as follows. By \cite{S} (or argue as in \cite[Lemma 1]{CC} via \cite[Theorem 8.1]{Agmon}), we have
\[
\|D\rho\|_{L^{\sigma'}(\Omega)}\leq C\left(\|b\rho\|_{L^{\sigma'}(\Omega)}+\|\rho\|_{L^{\sigma'}(\Omega)}+\|\psi\|_{W^{-1,\sigma'}(\Omega)}\right),
\]
where $C$ depends on $\lambda,\Omega,\sigma$. We first handle the last term on the right-hand side, observing that
\[
\int_\Omega\psi\varphi\,dx\leq \|\psi\|_{L^1(\Omega)}\|\varphi\|_{L^\infty(\Omega)}\leq C\|\psi\|_{L^1(\Omega)}\|\varphi\|_{W^{1,\sigma}_0(\Omega)}\ ,\sigma>d\ ,
\]
so that by \cite[Section 3.13]{AF} we get $\|\psi\|_{W^{-1,\sigma'}(\Omega)}\leq C\|\psi\|_{L^1(\Omega)}$ for some positive constant $C>0$. Then, we have
\begin{multline*}
\begin{aligned}
\|D\rho\|_{L^{\sigma'}(\Omega)}&\leq C\left(\|b\rho\|_{L^{\sigma'}(\Omega)}+\|\rho\|_{L^{\sigma'}(\Omega)}+\|\psi\|_{L^1(\Omega)}\right)\\
&= C\left(\|b\rho^\frac1k\rho^\frac{1}{k'}\|_{L^{\sigma'}(\Omega)}+\|\rho\|_{L^{\sigma'}(\Omega)}+\|\psi\|_{L^1(\Omega)}\right)\\
&\leq C\left(\left(\int_\Omega |b|^k\rho\,dx\right)^\frac1k\|\rho\|_{L^z(\Omega)}^{\frac{1}{k'}}+\|\rho\|_{L^{\sigma'}(\Omega)}+\|\psi\|_{L^1(\Omega)}\right)\\
&\leq C\left(\frac{1}{2\delta}\int_\Omega |b|^k\rho\,dx+\frac{\delta}{2}\|\rho\|_{L^z(\Omega)}+\|\rho\|_{L^{\sigma'}(\Omega)}+\|\psi\|_{L^1(\Omega)}\right),
\end{aligned}
\end{multline*}
where we applied first the H\"older's inequality for an exponent $z>\sigma'$ satisfying
\begin{equation}\label{cond}
\frac{1}{\sigma'}=\frac1k+\frac{1}{zk'}\ ,
\end{equation}
and then the generalized Young's inequality. By the interpolation inequalities we have
\[
\|\rho\|_{L^{\sigma'}(\Omega)}\leq \|\rho\|_{L^1(\Omega)}^{1-\theta}\|\rho\|_{L^z(\Omega)}^\theta\ ,\theta\in(0,1)\ ,\frac{1}{\sigma'}=1-\theta+\frac{\theta}{z}\ .
\]
Therefore, exploiting the fact that $\|\rho\|_{L^1(\Omega)}\leq \frac{\|\psi\|_{L^1(\Omega)}}{\lambda}$ (see \cite{BocJDE} or \eqref{rhoL1}) and applying once more the generalized Young's inequality to the above term, we get
\[
\|D\rho\|_{L^{\sigma'}(\Omega)}\leq \tilde C\left(\frac{1}{\delta}\int_\Omega |b|^k\rho\,dx+\delta\|\rho\|_{L^z(\Omega)}+\|\psi\|_{L^1(\Omega)}\right)\ ,
\]
where $\tilde C$ now depends also on $\lambda$. The conditions $k=1+\frac{d}{\sigma}$ and \eqref{cond} lead to 
\[
\frac{1}{z}=\frac{1}{\sigma'}-\frac1d\ ,
\]
so that the Sobolev embedding applies to obtain
\[
\|\rho\|_{L^z(\Omega)}\leq C_1\|D\rho\|_{L^{\sigma'}(\Omega)}\ .
\]
This implies, by choosing $\delta=\frac{1}{2\tilde C C_1}$, a bound on $\rho\in L^z(\Omega)$, and hence the assertion. 
\end{proof}
\begin{cor}\label{corfpreg}
Let $\rho$ be the nonnegative weak solution to \eqref{fp}. There exists a constant $C>0$ depending on $d,q,|\Omega|$ and not on $\lambda$ such that if $\frac{2d}{d+2}<q<\frac{d}{2}$ we have
\[
\|\rho\|_{L^{q'}(\Omega)}\leq C\left(\int_\Omega |b|^\frac{d}{q}\rho\,dx+\|\psi\|_{L^{p'}(\Omega)} \right)
\]
with $p=\frac{dq}{d-2q}$. If $q=\frac{d}{2}$ or $q>\frac{d}{2}$ we have the same estimate for any finite $p'<\infty$ and $p'=1$ respectively, but the constant of the estimate depends also on $\lambda$. 
\end{cor}
\begin{proof}
The proof follows from Proposition \ref{regfp1} applied with $k=\frac{d}{q}$ and the embedding of $W_0^{1,\sigma'}(\Omega)$ onto $L^{q'}(\Omega)$.
\end{proof}

\subsubsection{Local Sobolev regularity}
In this section we prove a local counterpart of the regularity results of the previous section for weak solutions to
\begin{equation}\label{adjprop}
-\Delta \rho+\lambda\rho+\mathrm{div}(b(x)\rho)=\psi(x)\text{ in }\Omega\ ,
\end{equation}
focusing only on the case $b\in L^k(\rho)$ for  $k=1+\frac{d}{\sigma}$, $\frac{d}{d-1}<\sigma'<d$. This gives an alternative proof of \cite[Theorem 1-(ii)]{BKR} without using elliptic regularity theory, and provides an explicit estimate on the size of the norm.
\begin{prop}\label{regadjloc}
Let $\rho\in W^{1,2}_{\mathrm{loc}}(\Omega)$ be a weak solution to \eqref{adjprop} and let $B_R=\{x\in\R^d:|x|<R\}$. Let $\frac{d}{d-1}<\sigma'<d$, $\psi\in L^{p'}(B_R)$, $p=\frac{d}{k-2}$, with $|b|\in L^{k}_{loc}(B_R;\rho\,dx)$ where $k=1+\frac{d}{\sigma}$ satisfies $2<k<1+\frac{d}{2}$. Then $\rho\in W^{1,\sigma'}_{\mathrm{loc}}(\Omega)$, and every weak solution to \eqref{adjprop}  satisfies the interior estimate
\[
\|\rho\|_{L^{\frac{d}{d-k}}(B_{\frac{R}{2}})}+\||D\rho|\|_{L^{\sigma'}(B_{\frac{R}{2}})}\leq C\left(\||b|\|_{L^k_{\mathrm{loc}}(B_R;\rho\,dx)} +\|\psi\|_{L^{p'}(B_R)}+1\right)
\]
where $C$ depends in particular on $d,\sigma',k,R,\|\rho\|_{L^1(B_R)}$.
\end{prop}
\begin{proof}
Let $\zeta\in C_0^\infty(B_R)$ be such that $0\leq \zeta\leq 1$ satisfying $\zeta>0$ in $B:=B_R$, $\zeta=1$ on the twice smaller ball $B_{R/2}$ and assume
\begin{equation}\label{cut}
\sup_x |D\zeta|\zeta^{-\eta}\leq C_\zeta
\end{equation}
for $\eta= \frac{\beta+1-2/k}{\beta+1}\in(0,1)$, $\beta=\frac{k-2}{d-k}$ and some positive constant $C_\zeta$. Such conditions are verified by $\zeta(x)=\psi(|x|/R)$, $\psi\in C_0^\infty(\R)$ with $\psi$ such that $0\leq\psi\leq 1$, $\psi(y)>0$ for $|y|<1$, $\psi(y)=0$ when $|y|\geq1$ and $\psi(y)=1$ on $|y|\leq\frac12$, with $\psi(y)=\mathrm{exp}((y^2-1)^{-1})$ near the end-points $-1$ and $1$, cf \cite[Theorem 1.7.4]{BKRS}. \\
We test the equation against $\varphi=\rho^\beta\zeta^2$, $\beta=\frac{k-2}{d-k}$ and set $\alpha=\frac{2(d-k)}{d-2}$. We have
\[
\int_B D\rho\cdot D(\zeta^2\rho^\beta)\,dx-\int_B \rho b\cdot D(\zeta^2\rho^\beta)\,dx+\lambda\int_B \rho^{\beta+1}\zeta^2\,dx=\int_B \zeta^2\rho^\beta \psi\,dx\ .
\]
We thus write
\begin{multline*}
\int_B |D\rho|^2\rho^{\beta-1}\zeta^2\,dx+\frac{\lambda}{\beta}\int_B \rho^{\beta+1}\zeta^2\,dx\leq \frac{2}{\beta}\int_B |D\rho||D\zeta|\rho^\beta\zeta\,dx\\
+\int_B \zeta^2|b|\rho^\beta|D\rho|\,dx+\frac2\beta \int_B |D\zeta|\zeta\rho^{\beta+1}|b|\,dx+\frac1\beta\int_B \zeta^2\rho^\beta\psi\,dx\ .
\end{multline*}
We then use Young's inequality to get
\begin{multline}\label{mainineq}
\int_B |D\rho|^2\rho^{\beta-1}\zeta^2\,dx+\frac{\lambda}{\beta}\int_B \rho^{\beta+1}\zeta^2\,dx\leq \frac{2}{\beta}\int_B |D\rho||D\zeta|\rho^\beta\zeta\,dx\\
+\frac18\int_B|D\rho|^2\rho^{\beta-1}\zeta^2\,dx+2\int_B\rho^{\beta+1}|b|^2\zeta^2\,dx  +\frac2\beta \int_B |D\zeta|\zeta\rho^{\beta+1}|b|\,dx+\frac1\beta\int_B \zeta^2\rho^\beta\psi\,dx\\
=\mathrm{(I)+(II)+(III)+(IV)+(V)}\ .
\end{multline}
Note that $\mathrm{(II)}$ can be absorbed on the left-hand side.
We observe that by the Sobolev's inequality \cite[Theorem 7.10]{GT}
\begin{multline}\label{sob}
\frac12\int_B |D\rho|^2\rho^{\beta-1}\zeta^2\,dx=c_\beta\int_B |D\rho^{\frac{\beta+1}{2}}|^2\zeta^2\,dx=c_\beta\int_B |D(\rho^{\frac{\beta+1}{2}}\zeta)|^2\,dx-c_\beta\int_B \rho^{\beta+1}|D\zeta|^2\,dx\\
\geq c_{d,\beta}\left(\int_B \rho^{(\beta+1)\frac{d}{d-2}}\zeta^{\frac{2d}{d-2}}\,dx\right)^{1-\frac2d}-c_\beta\int_B \rho^{\beta+1}|D\zeta|^2\,dx\ ,
\end{multline}
where $c_\beta=\frac{2}{(\beta+1)^2}$. We now start estimating the terms in \eqref{mainineq} (and the one on the right-hand side of the above inequality, which has a negative sign). Applying the H\"older's inequality with exponents $\left(\frac{k}{2},\frac{k}{k-2}\right)$ and the weighted Young's inequality with the pair $\left(\frac{(d-2)k}{2(d-k)},\frac{(d-2)k}{d(k-2)}\right)$, we get
\begin{multline*}
\mathrm{(III)}=\int_B\zeta^2\rho^{\beta+1}|b|^2\,dx=\int_B |b|^2\rho^{\frac2k}\zeta^{\frac{2\alpha}{k}}\rho^{\beta+1-\frac{2}{k}}\zeta^{\frac{2d(k-2)}{k(d-2)}}\,dx\\
\leq  \left(\int_B\zeta^{\frac{2(d-k)}{d-2}} |b|^k\rho\right)^{\frac2k}\left(\int_B \rho^{\beta\frac{k}{k-2}+1}\zeta^{\frac{2d}{d-2}}\,dx\right)^{1-\frac2k}\\
\leq C_1(d,\beta,k)\left(\int_B\zeta^{\frac{2(d-k)}{d-2}} |b|^k\rho\,dx\right)^{\frac{d-2}{d-k}}+\frac{c_{d,\beta}}{24}\left(\int_B \rho^{\beta\frac{k}{k-2}+1}\zeta^{\frac{2d}{d-2}}\,dx\right)^{1-\frac2d}.
\end{multline*}
Owing to the above estimate, we then obtain via the Young's inequality
\begin{multline*}
\mathrm{(IV)}=\frac2\beta \int_B |D\zeta|\zeta\rho^{\beta+1}|b|\,dx\leq \frac{1}{\beta}\int_B \rho^{\beta+1}|D\zeta|^2\,dx+\frac{1}{\beta}\int_B |b|^2\zeta^2\rho^{\beta+1}\,dx\\
\leq \frac{1}{\beta}\int_B \rho^{\beta+1}|D\zeta|^2\,dx+\frac{1}{\beta}\left(\int_B\zeta^{\frac{2(d-k)}{d-2}} |b|^k\rho\right)^{\frac2k}\left(\int_B \rho^{\beta\frac{k}{k-2}+1}\zeta^{\frac{2d}{d-2}}\,dx\right)^{1-\frac2k}\\
\leq \frac{1}{\beta}\int_B \rho^{\beta+1}|D\zeta|^2\,dx+C_2(d,k,\beta)\left(\int_B\zeta^{\frac{2(d-k)}{d-2}} |b|^k\rho\,dx\right)^{\frac{d-2}{d-k}}+\frac{c_{d,\beta}}{24}\left(\int_B \rho^{\beta\frac{k}{k-2}+1}\zeta^{\frac{2d}{d-2}}\,dx\right)^{1-\frac2d}\ .
\end{multline*}
We now estimate $\int_B \rho^{\beta+1}|D\zeta|^2\,dx$ (which appeared in \eqref{sob} and from the above inequality) using \eqref{cut} for $\eta:=\frac{\beta+1-2/k}{\beta+1}\in(0,1)$, and H\"older's inequality applied with the exponents $\xi=\frac{d}{(d-2)\eta}$ and $\xi'$. In particular, using the definition of $\beta$ and $\eta$ one first checks that $2\xi'/k=1$. Indeed, 
\[
\frac{1}{\xi}=\frac{d-2}{d}\eta
\]
hence
\[
\frac{1}{\xi'}=1-\frac{d-2}{d}\eta.
\]
Therefore
\[
\frac{k}{2\xi'}=\frac{k}{2}-\frac{k}{2}\frac{d-2}{d}\frac{\beta+1-\frac2k}{\beta+1}=\frac{k}{2}-\frac{k}{2}\frac{d-2}{d}\frac{d(k-2)}{d-2}\frac{d-k}{k(d-k)}=\frac{k}{2}-\frac{k-2}{2}=1.
\]
We use once more the H\"older's inequality first and then the Young's inequality, together with \eqref{cut}, to conclude 
\begin{multline*}
\begin{aligned}
\int_B \rho^{\beta+1}|D\zeta|^2\,dx&\leq \frac{C_\zeta}{\beta}\int_B\rho^{\beta+1-\frac2k}\zeta^{2\eta}\rho^{\frac2k}\,dx\\
&\leq \frac{C_\zeta}{\beta}\left(\int_B\rho^{(\beta+1)\frac{d}{d-2}}\zeta^{\frac{2d}{d-2}}\,dx\right)^{\left(1-\frac2d\right)\eta}\left(\int_B\rho^{\frac{2\xi'}{k}}\,dx\right)^\frac{1}{\xi'}\\
&\leq \frac{c_{d,\beta}}{24} \left(\int_B\rho^{(\beta+1)\frac{d}{d-2}}\zeta^{\frac{2d}{d-2}}\,dx\right)^{1-\frac2d}+C(d,\beta,\zeta, \|\rho\|_{L^1(B_1)})\ ,
\end{aligned}
\end{multline*}
where we also used that $\frac{k}{2\xi'}=1$.
Moreover, using that $\frac{2d}{d-2}< \frac{2d}{k-2}$ (note that $k<1+\frac{d}{2}$, hence $k<d$ for $d>2$), we write, applying the H\"older and Young's inequalities
\begin{multline*}
\mathrm{(V)}=\frac1\beta\int_B \zeta^2\rho^\beta\psi\,dx\leq \frac1\beta \|\psi\|_{L^{p'}(B)}\left(\int_B \rho^{\beta p}\zeta^{\frac{2d}{k-2}}\,dx\right)^{\frac1p}\leq \frac1\beta \|\psi\|_{L^{p'}(B)}\left(\int_B \rho^{\beta p}\zeta^{\frac{2d}{d-2}}\,dx\right)^{\frac1p}\\
\leq \frac{c_{d,\beta}}{24}\left(\int_B \rho^{\beta p}\zeta^{\frac{2d}{d-2}}\,dx\right)^{1-\frac2d}+C(\beta,p,c_{d,\beta})\|\psi\|_{L^{p'}(B)}^{\frac{d-2}{d-k}}\ .
\end{multline*}
We now consider
\begin{multline*}
\mathrm{(I)}=\frac{2}{\beta}\int_B |D\rho||D\zeta|\rho^\beta\zeta\,dx\leq \frac18 \int_B |D\rho|^2\rho^{\beta-1}\zeta^2\,dx+C(\beta)\int_B \rho^{\beta+1}|D\zeta|^2\,dx\\
\leq \frac18 \int_B |D\rho|^2\rho^{\beta-1}\zeta^2\,dx+\frac{c_{d,\beta}}{24} \left(\int_B\rho^{(\beta+1)\frac{d}{d-2}}\zeta^{\frac{2d}{d-2}}\,dx\right)^{1-\frac2d}+C(d,\beta,\zeta, \|\rho\|_{L^1(B_1)})\ .
\end{multline*}

We plug all the estimates together in \eqref{mainineq} noting that \[
\frac{d}{d-k}=(\beta+1)\frac{d}{d-2}=\beta\frac{k}{k-2}+1=\beta p\ ,
\] to obtain
\begin{multline*}
\begin{aligned}
\frac{c_{d,\beta}}{4}\left(\int_B \rho^{(\beta+1)\frac{d}{d-2}}\zeta^{\frac{2d}{d-2}}\,dx\right)^{1-2/d}&+\frac14\int_B \rho^{\beta-1}|D\rho|^2\zeta^2\,dx\\
&\leq C_3\left[ \left(\int_B \zeta^{\frac{2(d-k)}{d-2}}|b|^k\rho\,dx\right)^{\frac{d-2}{d-k}}+\|\psi\|_{L^{q'}(B)}^{\frac{d-2}{d-k}}+1\right]\\
\\
&\leq C_4\left[\int_B \zeta^{\frac{2(d-k)}{d-2}}|b|^k\rho\,dx+\|\psi\|_{L^{p'}(B)}+1\right]^{\frac{d-2}{d-k}}\ ,
\end{aligned}
\end{multline*}
where $C_3,C_4$ depends on $d,\beta,C_\zeta,k$ together with $\|\rho\|_{L^1(B_R)}$. This in turn allows to conclude for some positive constant $C_5>0$
\[
\left(\int_B \rho^{\frac{d}{d-k}}\zeta^{\frac{2d}{d-2}}\,dx\right)^{1-2/d}
\leq C_5\left[\left(\int_B \zeta^{\frac{2(d-k)}{d-2}} |b|^k\rho\,dx\right)+\|\psi\|_{L^{p'}(B)}+1\right]^{\frac{d-2}{d-k}}\ .
\]
This in particular yields the estimate $\|\rho\zeta^{\frac{2(d-k)}{d-2}}\|_{L^{\frac{d}{d-k}}(B_R)}\geq \|\rho\|_{L^{\frac{d}{d-k}}(B_\frac{R}{2})}$. Finally, by the H\"older's inequality, using that $2<k<1+\frac{d}{2}$, we have
\begin{multline*}
\|\zeta^{\frac{2(d-k)}{d-2}}D\rho\|_{L^{\frac{d}{d-k+1}}(B)}\leq \|\zeta\rho^{\frac{\beta-1}{2}}D\rho\|_{L^2(B)}\|\rho^{\frac{1-\beta}{2}}\zeta^{\frac{d-2k+2}{d-2}}\|_{L^{\frac{2d}{d-2k+2}}(B)}\\
=\left(\int_B \rho^{\beta-1}\zeta^2|D\rho|^2\,dx\right)^\frac12\left(\int_{B}\rho^{\frac{d}{d-k}}\zeta^{\frac{2d}{d-2}}\,dx\right)^{\frac{d-2k+2}{2d}}\ ,
\end{multline*}
and we conclude the local gradient regularity using the previous estimates.
\end{proof}

\subsection{Integral estimates for viscous Hamilton-Jacobi equations}\label{estu}
In this section we focus on strong solutions in the standard sense of \cite[Chapter 9]{GT} of \eqref{hj} such that $u\in W^{2,q}(\Omega)\cap W^{1,q}_0(\Omega)$, $q>\frac{d}{\gamma'}$, dealing with the case $\gamma>\frac{d+2}{d}$. On one hand, we observe that under the restriction $q\geq\frac{d}{\gamma'}$ we have the embedding
\[
W^{2,q}\hookrightarrow W^{1,q\gamma}
\] 
and hence $b\sim |Du|^{\gamma-1}\in L^p$, $p\geq d$. Therefore, when the velocity field $b=-D_pH(Du)$, under the standing growth conditions on the Hamiltonian term $H$, the adjoint equation \eqref{fp} turns out to be well-posed by Proposition \ref{well}. The above restriction on $\gamma$ is imposed because we use the energy formulation of \eqref{hj} via the embedding
\[
W^{2,q}\hookrightarrow W^{1,2}\ ,
\]
which occurs when $q\geq \frac{2d}{d+2}$. Then, $\frac{d}{\gamma'}\geq \frac{2d}{d+2}$ precisely when $\gamma\geq \frac{d+2}{d}$, which is indeed the critical threshold guaranteeing the validity of the energy formulation of the problem, see also \cite{GMP14}. One expects to address the case below $\gamma= \frac{d+2}{d}$ using different techniques and notion of solutions, cf \cite{MercaldoCPAA,GMP14} and the references therein.\\
In this section, we denote by $L$ the Lagrangian of $H$, defined as its Legendre transform, i.e. $L(x,\nu)=\sup_{p\in\R^d}\{p\cdot \nu-H(x,p)\}$. Moreover, by the convexity property of $H$ we have
\[
H(x,p)=\sup_{\nu\in\R^d}\{p\cdot \nu-L(x,\nu)\}
\]
and
\[
H(x,p)=p\cdot \nu-L(x,\nu)\text{ if and only if }\nu=D_pH(x,p)\ .
\]
Under the standing assumptions we will henceforth use that for some $C_L>0$ depending on $C_H$
\begin{equation}\label{L}\tag{L}
C_L^{-1}|\nu|^{\gamma'}-C_L\leq L(x,\nu)\leq C_L|\nu|^{\gamma'}+C_L\ ,\gamma'=\frac{\gamma}{\gamma-1}\ .
\end{equation}
We will focus here on the case $\frac{d}{\gamma'}<q<\frac{d}{2}$, i.e. when $\gamma<2$, and deduce $L^p$ integral estimates in the next sections.
\subsubsection{The subcritical case $q>\frac{d}{\gamma'}$}
We start with the following bound on the positive part of $u$, $u^+=\max\{u,0\}$. It holds for any $q\geq \frac{2d}{d+2}$. Below, $u^-=(-u)^+$, while $T_k(s)=\max\{-s,\min\{s,k\}\}$ will denote the standard truncation operator at level $k>0$.  
\begin{prop}\label{upiu}
Assume that $H$ is nonnegative and let $u\in W^{2,q}(\Omega)\cap W^{1,q}_0(\Omega)$, $q\geq \frac{2d}{d+2}$, be a strong solution to the Dirichlet problem \eqref{hj}. There exists a positive constant $C_0$ such that 
\[
\|u^+\|_{L^p(\Omega)}\leq C_0\|f^+\|_{L^q(\Omega)}
\]
with $p=\frac{dq}{d-2q}$ if $q<\frac{d}{2}$, any finite $p$ when $q=\frac{d}{2}$, $p=\infty$ if $q>\frac{d}{2}$. Here, $C_0$ depends on $d,q, |\Omega|$ but not on $\lambda$.
\end{prop}
\begin{proof}
We first consider the case $q<\frac{d}{2}$. For $k>0$ consider the weak nonnegative solution $\mu=\mu_k$ to the Dirichlet problem
\begin{equation}\label{aux1}
\begin{cases}
-\Delta \mu+\lambda \mu= \psi_1(x)&\text{ in }\Omega\ ,\\
\mu=0&\text{ on }\partial\Omega\ .
\end{cases}
\end{equation}
where 
\[
\psi_1(x)=\frac{[T_k(u^+(x))]^{p-1}}{\|u^+\|_{L^{p}(\Omega)}^{p-1}}.
\]
Note that $\|\psi_1\|_{p'}\leq 1$. By Corollary \ref{corfpreg} with $b\equiv0$ we deduce 
\[
\|\mu\|_{L^{q'}(\Omega)}\leq C
\]
with $C$ independent of $k$. By Kato's inequality \cite{BrezisAMO}, $u^+$ is a weak subsolution to
\[
-\Delta u^++\lambda u^+\leq [f(x)-H(x,Du(x))]\mathbbm{1}_{\{u>0\}}\text{ in }\Omega\ .
\]
Then, using $\mu$ as a test function in the previous equation and $u^+$ as a test function in problem \eqref{aux1}, we deduce
\begin{multline*}
\int_\Omega \psi_1(x)u^+(x)\,dx\leq \int_{\Omega\cap\{u>0\}}f\mu\,dx-\int_\Omega H(x,Du)\mathbbm{1}_{\{u>0\}}\mu\,dx\\
\leq \|f^+\|_{L^q(\Omega)}\|\mu\|_{L^{q'}(\Omega)}\leq C\|f^+\|_{L^q(\Omega)}\ .
\end{multline*}
The estimate follows by duality using the definition of $\psi_1$, applying H\"older's inequality on the right-hand side of the above inequality and sending $k\to\infty$. We now briefly discuss the case $q>\frac{d}{2}$. First, one considers \eqref{aux1} with a right-hand side $\psi_1\geq0$ such that $\|\psi_1\|_{L^1(\Omega)}=1$. Using the same strategy, without need to use a truncation argument, using that $\mu$ is nonnegative and $H\geq0$, one obtains $u^+\in L^\infty(\Omega)$ using classical estimates for boundary-value problems, see \cite{BGae}. We refer to \cite[Proposition 3.7]{cg20} for further details, being the proof similar.
\end{proof}
We now consider the more delicate bound on the negative part $u^-$, which needs the extra integrability requirement $q>\frac{d}{\gamma'}$.
\begin{prop}\label{umeno}
Assume that \eqref{H} holds. Let $u\in W^{2,q}(\Omega)\cap W^{1,q}_0(\Omega)$ be a strong solution to \eqref{hj} with $q>d/\gamma'$. There exists a positive constant $C$ depending on $d,q,C_H,\gamma$ such that
\[
\|u^-\|_{L^p(\Omega)}\leq C\|f^-\|_{L^q(\Omega)}^{\frac{q\gamma'}{q\gamma'-d}}+CC_L\frac{|\Omega|^\frac1p}{\lambda}+C\|f^-\|_{L^q(\Omega)}
\]
where $p=\frac{dq}{d-2q}$ when $q<\frac{d}{2}$, any $p<\infty$ when $q=\frac{d}{2}$ and $p=\infty$ when $q>\frac{d}{2}$.
\end{prop}

\begin{proof}
For $q<\frac{d}{2}$, $k>0$, we define as $\rho=\rho_k$ the weak nonnegative solution to the Dirichlet problem
\begin{equation}\label{adjk}
\begin{cases}
-\Delta \rho(x)+\lambda \rho(x)+\mathrm{div}(D_pH(x,Du(x))\mathbbm{1}_{\{u<0\}}\rho(x))=\psi_2(x) &\text{ in }\Omega\ ,\\
\rho=0 &\text{ on }\partial\Omega\ .
\end{cases}
\end{equation}
where $\psi_2(x)=\frac{[T_k(u^-(x))]^{p-1}}{\|u^-\|_{L^p(\Omega)}^{p-1}}$. By duality as in the case of Proposition \ref{upiu}, using again Kato's inequality as in \cite[p. 148]{GMP14}, \cite{CGpar}, we get
\[
\int_\Omega u^-(x)\psi_2(x)\,dx+\int_\Omega \left[-D_pH(x,Du)\cdot Du^--H(x,Du)\right]\mathbbm{1}_{\{u<0\}}\rho\,dx\leq \int_{\Omega\cap \{u<0\}}f\rho\,dx\ .
\]
By the properties of the Lagrangian \eqref{L}, we estimate
\begin{multline*}
\left[-D_pH(x,Du)\cdot Du^--H(x,Du)\right]\mathbbm{1}_{\{u<0\}}=L(x,D_pH(x,-Du^-))\mathbbm{1}_{\{u<0\}}\\
\geq [C_L^{-1}|D_pH(x,Du)|^{\gamma'}-C_L]\mathbbm{1}_{\{u<0\}}\ .
\end{multline*}
Therefore, we have by Corollary \ref{corfpreg} (note that it can be safely applied since $\gamma>\frac{d+2}{d}$) and the fact that the condition $\|\psi_2\|_{p'}\leq1$ implies that $\|\rho\|_{1}\leq \frac{|\Omega|^\frac1p}{\lambda}$ through \eqref{rhoL1}
\begin{multline*}
\int_\Omega u^-(x)\psi_2(x)\,dx+C_L^{-1}\int_\Omega |D_pH(x,Du)|^{\gamma'}\mathbbm{1}_{\{u<0\}}\rho\,dx\leq C_L\int_\Omega \rho\mathbbm{1}_{\{u<0\}}+\int_\Omega f^-\rho\,dx\\
\leq C_L\frac{|\Omega|^\frac1p}{\lambda}+\|f^-\|_{L^q(\Omega)}\|\rho\|_{L^{q'}(\Omega)}\leq C_L\frac{|\Omega|^\frac1p}{\lambda}+C_1\|f^-\|_{L^q(\Omega)}\left(\int_\Omega |D_pH(x,Du)|^{\frac{d}{q}}\mathbbm{1}_{\{u<0\}}\rho\,dx+1\right).
\end{multline*}
Since $q>d/\gamma'$, we can absorb the second term in the above right-hand side via Young's inequality to obtain
\[
\int_\Omega u^-(x)\psi_2(x)\,dx\leq C_L\frac{|\Omega|^\frac1p}{\lambda}+C_2\left(\|f^-\|_{L^q(\Omega)}^{\frac{q\gamma'}{q\gamma'-d}}+\frac{|\Omega|^\frac1p}{\lambda}+\|f^-\|_{L^q(\Omega)}\right)
\]
and then let $k\to\infty$ to conclude the estimate. The case $q=\frac{d}{2}$ is similar, while the case $q>\frac{d}{2}$ requires to take a nonnegative $\psi=\psi_2\in L^1(\Omega)$ with $\|\psi_2\|_{L^1(\Omega)}=1$ to conclude that $u^-\in L^\infty(\Omega)$ via the same duality argument and exploiting that
\[
\int_\Omega \rho\leq \frac{\|\psi_2\|_{L^1(\Omega)}}{\lambda}=\frac{1}{\lambda}\ .
\]
Indeed, fix $q>\frac{d}{\gamma'}$ and start again with  
\begin{multline*}
\int_\Omega u^-(x)\psi_2(x)\,dx+C_L^{-1}\int_\Omega |D_pH(x,Du)|^{\gamma'}\mathbbm{1}_{\{u<0\}}\rho\,dx\leq C_L\int_\Omega \rho\mathbbm{1}_{\{u<0\}}+\int_\Omega f^-\rho\,dx\\
\leq \frac{C_L}{\lambda}+\|f^-\|_{L^q(\Omega)}\|\rho\|_{L^{q'}(\Omega)}.
\end{multline*}
One then chooses $\sigma$ such that
\[
\frac{1}{q'}=\frac1\sigma-\frac1d
\]
so that taking $q>\frac{d}{2}$ we get the condition $\sigma<\frac{d}{d-1}$. We can then apply Corollary \ref{corfpreg} to get
\[
C_L^{-1}\int_\Omega |D_pH(x,Du)|^{\gamma'}\mathbbm{1}_{\{u<0\}}\rho\,dx\leq \frac{C_L}{\lambda}+C\|f^-\|_{L^q(\Omega)}\left(\int_\Omega |D_pH(x,Du)|^{k}\mathbbm{1}_{\{u<0\}} \rho\,dx+1\right)
\]
with $k=1+\frac{d}{\sigma}=\frac{d}{q}<\gamma'$, and then apply the Young's inequality to conclude the assertion.
\end{proof}
Combining Propositions \ref{upiu} and \ref{umeno} we get the following result under the assumption $q>\frac{d}{\gamma'}$ (recall that $\frac{d}{\gamma'}>\frac{2d}{d+2}$ under the standing restrictions on $\gamma$).
\begin{cor}\label{finalint}
Assume that \eqref{H} holds. Let $u\in W^{2,q}(\Omega)\cap W^{1,q}_0(\Omega)$ be a strong solution to \eqref{hj} with $q>d/\gamma'$. Assume there exists $K>0$ such that $\|f\|_{L^q(\Omega)}\leq K$. There exists a positive constant $C$ depending on $K,d,q,C_H,|\Omega|,\lambda$ such that
\[
\|u\|_{L^p(\Omega)}\leq C
\]
where 
\begin{itemize}
\item $p=\infty$ when $q>\frac{d}{2}$ and any $\gamma>\frac{d+2}{d}$;
\item any $p<\infty$ when $q=\frac{d}{2}>\frac{d}{\gamma'}$ and $\gamma>\frac{d+2}{d}$;
\item $p=\frac{dq}{d-2q}$ when $\frac{d}{\gamma'}<q<\frac{d}{2}$ (which implies $\gamma<2$) and $\gamma>\frac{d+2}{d}$. 
\end{itemize}
\end{cor}
\begin{proof}
It readily follows from Propositions \ref{upiu} and \ref{umeno}.
\end{proof}
\begin{rem}
An inspection of the proofs suggest that the previous results can be stated for energy solutions, as in \cite{cg20}, weakening thus the a priori requirement $u\in W^{2,q}$.
\end{rem}
\begin{rem}
$L^\infty$ bounds for distributional subsolutions in $W_0^{1,\gamma}$ in the case $\gamma>2$ already appeared in \cite[Theorem 5.2]{DAP} through a Stampacchia-type argument under the restriction $f\in L^q$, $q>\frac{d}{\gamma}$ (note that when $\gamma>2$ we have $\frac{d}{\gamma}<\frac{d}{2}$). When $\gamma=2$ the $L^\infty$ bounds for solutions belonging a priori to $W_0^{1,2}\cap L^\infty$ were proved in \cite[Theorem 2.1]{BMP92} again via a Stampacchia's method. Instead, our $L^\infty$ bound from Corollary \ref{finalint} holds for any $\gamma>1+\frac{2}{d}$, and thus provide a unified proof for this end-point integral estimate. We further remark that, in general, the $L^\infty$ estimate is false when $\gamma\leq 2$ if $u$ does not belong to a suitable class of solutions, cf \cite[Remark 5.1]{DAP} or \cite[Remark 2.1]{BMP92}. Indeed, Example 1.3 in \cite{GMP14} shows that there exists an unbounded weak solution for problems with subquadratic growth in the gradient if $u$ does not satisfy the condition
\[
|u|^\tau\in W^{1,2}_0(\Omega)\ ,\tau=\frac{(d-2)(\gamma-1)}{2(2-\gamma)}\ .
\]
We observe that in the limiting case $q=\frac{d}{\gamma'}$, for $\gamma<2$, the membership $u\in W^{2,q}$ leads to the above condition when $\gamma<2$. Indeed, if $u\in W^{2,q}$ for $q=\frac{d}{\gamma'}$, it follows by Sobolev embeddings that $u\in W^{1,\frac{dq}{d-q}}$, i.e. $u\in W^{1,d(\gamma-1)}$. This in particular implies, applying again the Sobolev inequality, that $u\in L^{\frac{d(\gamma-1)}{2-\gamma}}$, i.e. $|u|^\tau\in L^{2^*}\hookrightarrow L^2$. Finally, $D|u|^\tau\in L^2$ when $|u|^{\tau-1}|Du|\in L^2$, which is true by the H\"older's inequality using that $u\in L^{\frac{d(\gamma-1)}{2-\gamma}}$ together with $Du\in L^{d(\gamma-1)}$. Finally, for the case of natural gradient growth $\gamma=2$, the Sobolev embedding for $q>\frac{d}{2}$ implies that $u\in L^\infty$, and hence we are in the same situation of \cite[Remark 2.1]{BMP92}, where solutions belong a priori to $L^\infty$.
\end{rem}
\subsubsection{The end-point threshold $q=\frac{d}{\gamma'}$}
The results of the previous section do not encompass the critical integrability value $q=\frac{d}{\gamma'}$. In this case one expects integrability estimates to depend on finer properties than the sole control of $f\in L^q$, see e.g. \cite{GMP14}. Aim of this section is to derive integrability estimates in this end-point situation via duality methods, together with establishing new stability properties in Lebesgue spaces for strong solutions. The approach of this section is inspired by that appeared in the context of parabolic problems in \cite[Corollary 3.4]{CGpar}.\\
Consider for $\lambda>0$ the Dirichlet problem with truncated right-hand side datum
\begin{equation}\label{hjtrunc}
\begin{cases}
- \Delta u_k(x) + H(x, Du_k(x)) +\lambda u_k(x) = T_k(f(x)) & \text{in $\Omega$} \\
u_k(x) = 0 & \text{on $\partial\Omega$,}
\end{cases}
\end{equation}
Note that now the right-hand side of the equation belongs to $L^\infty(\Omega)$, and existence and uniqueness of strong solutions in $W^{2,q}$ for the Dirichlet problem is guaranteed by the results proved in e.g. \cite[Lemma 3]{AC}, see also \cite{MaSoBook,KK}.

\begin{prop}\label{stabLpest}
Assume that \eqref{H} holds. Let $u$ and $u_k$ be strong solutions to \eqref{hj} and \eqref{hjtrunc} respectively, $\gamma<2$ and $q=\frac{d}{\gamma'}$. Then, there exists a constant $C$ depending on $f,C_H,q,d,|\Omega|, \lambda$ such that
\begin{equation}\label{stabLp}
\|u-u_k\|_{L^p(\Omega)}\leq C\|f-T_k(f)\|_{L^q(\Omega)}
\end{equation}
where $p=\frac{dq}{d-2q}=\frac{d(\gamma-1)}{2-\gamma}$.
\end{prop}
\begin{proof}
Let $w=u-u_k$. Here, $w$ depends on $k$, but we will drop the subscript $k$. Let $\rho$ be the weak solution to the Dirichlet problem
\begin{equation}\label{pbl3}
\begin{cases}
-\Delta \rho(x)+\lambda \rho(x)-\mathrm{div}(D_pH(x,Du_k(x))\rho(x))=\psi_3(x) &\text{ in }\Omega\ ,\\
\rho=0 &\text{ on }\partial\Omega\ .
\end{cases}
\end{equation}
where
\[
\psi_3(x)=\frac{[w^+(x)]^{p-1}}{\|w^+\|_{L^p(\Omega)}^{p-1}}\ .
\]
As in Propositions \ref{upiu} and \ref{umeno}, one should truncate the right-hand side term $\psi_3$ to ensure the existence of solutions on energy spaces, and then pass to the limit. We will drop this step for simplicity.
We first proceed by proving a bound on $\int_\Omega |D_pH(x,Du_k))|^{\gamma'}\rho\,dx$. By duality arguments, testing \eqref{hjtrunc} with $\rho$ and \eqref{pbl3} with $u_k$ we have
\begin{equation}\label{mainLstab}
\int_\Omega L(x,D_pH(x,Du_k))\rho\,dx= -\int_\Omega T_k(f(x))\rho(x)\,dx+\int_\Omega u_k(x)\psi_3(x)\ .
\end{equation}
For $h>0$ to be chosen we write by \eqref{rhoL1}
\begin{multline*}
-\int_\Omega T_k(f(x))\rho(x)\,dx\leq \int_\Omega f^-(x)\rho(x)\,dx\leq \int_{\Omega\cap\{f^-\geq h\}} f^-\rho\,dx+h\frac{|\Omega|^\frac1p}{\lambda}\\
\leq \|f^-\mathbbm{1}_{\{f^-\geq h\}}\|_{L^q(\Omega)}\|\rho\|_{L^{q'}(\Omega)}+h\frac{|\Omega|^\frac1p}{\lambda}\ .
\end{multline*}
We then use Proposition \ref{upiu} applied to $u_k$ to deduce
\[
\int_\Omega u_k(x)\psi_3(x)\leq \|u_k\|_{L^p(\Omega)}\|\psi_3\|_{L^{p'}(\Omega)}\leq C_0\|T_k(f^+)\|_{L^q(\Omega)}\leq C_0\|f^+\|_{L^q(\Omega)}\ .
\]
We then conclude
\begin{multline*}
C_L^{-1}\int_\Omega |D_pH(x,Du_k))|^{\gamma'}\rho\,dx-C_L\int_\Omega\rho\,dx\leq \frac{h|\Omega|^\frac1p}{\lambda}+C_0\|f^+\|_{L^q(\Omega)}\\
+\|f^-\mathbbm{1}_{\{f^-\geq h\}}\|_{L^q(\Omega)}\|\rho\|_{L^{q'}(\Omega)}\ .
\end{multline*}
Plugging the previous estimates on \eqref{mainLstab} it follows that
\begin{multline*}
C_L^{-1}\int_\Omega |D_pH(x,Du_k))|^{\gamma'}\rho\,dx\leq \frac{|\Omega|^\frac1p}{\lambda}(h+C_L)+C_0\|f^+\|_{L^q(\Omega)}\\
+C_1\|f^-\mathbbm{1}_{\{f^-\geq h\}}\|_{L^q(\Omega)}\left(\int_\Omega |D_pH(x,Du_k))|^{\gamma'}\rho\,dx+1\right)
\end{multline*}
with $C_1$ depending on $d,\gamma,\Omega$. We finally choose $h$ large enough to ensure
\[
\|f^-\mathbbm{1}_{\{f^-\geq h\}}\|_{L^q(\Omega)}\leq\frac{1}{2C_1C_L}
\]
and absorb the term on the right-hand side on the left-hand one, concluding the bound
\[
\int_\Omega |D_pH(x,Du_k))|^{\gamma'}\rho\,dx\leq 2C_L\left[\frac{|\Omega|^\frac1p}{\lambda}(h+C_L)+C_0\|f^+\|_{L^q(\Omega)}\right]+1=:\widetilde C\ .
\]
We are now in position to obtain the estimate on $w^+$. It is immediate to observe that by convexity of $H(x,\cdot)$, $w^+$ is a weak subsolution to
\[
-\Delta w^++\lambda w^++D_pH(x,Du_k(x))\cdot Dw^+\leq [f-T_k(f)]\mathbbm{1}_{\{w>0\}}\ .
\]
By duality we conclude
\[
\|w^+\|_{L^p(\Omega)}=\int_\Omega w^+(x)\psi_3(x)\,dx\leq \int_\Omega [f-T_k(f)]\mathbbm{1}_{\{w>0\}}\rho\,dx
\]
and thus by Corollary \ref{corfpreg}
\begin{multline*}
\|w^+\|_{L^p(\Omega)}\leq C_1\|f-T_k(f)\|_{L^{q}(\Omega)}\left(\int_\Omega |D_pH(x,Du_k))|^{\gamma'}\rho\,dx+1\right)\\
\leq C_1(\widetilde C+1)\|f-T_k(f)\|_{L^q(\Omega)}\ .
\end{multline*}
Finally, to obtain the integral estimates on $w^-$ (and hence conclude the desired estimate on $w$) one can proceed using the same scheme considering the dual problem
\begin{equation*}
\begin{cases}
-\Delta \hat \rho(x)+\lambda \hat \rho(x)-\mathrm{div}(D_pH(x,Du_k(x))\hat \rho(x))=\psi_4(x) &\text{ in }\Omega\ ,\\
\hat \rho=0 &\text{ on }\partial\Omega\ ,
\end{cases}
\end{equation*}
where now
\[
\psi_4(x)=\frac{[w^-(x)]^{p-1}}{\|w^-\|_{L^p(\Omega)}^{p-1}}\ .
\] 
\end{proof}
\begin{rem}\label{equi}
First, we observe that \eqref{stabLp} readily implies
\[
\|u\|_{L^p(\Omega)}\leq C
\]
when $p=\frac{d(\gamma-1)}{2-\gamma}$, $q=\frac{d}{\gamma'}$. However, the dependence of the constant $C$ in the above estimate is different with respect to the $L^p$ bound obtained in Corollary \ref{finalint}. Here, $C$ depends on $\|f^+\|_{L^q(\Omega)},h$ where $h$ has been chosen so that
\[
\|f^-\mathbbm{1}_{\{f^-\geq h\}}\|_{L^q(\Omega)}\leq\frac{1}{2C_1C_L}\ .
\]
It is worth observing that these constants remain bounded whenever $f$ varies in bounded and equi-integrable sets in $L^q$, cf Definition 2.23 and Theorem 2.29 in \cite{FL}.

\end{rem}
\begin{rem}
The estimate $u\in L^p$ with $p=\frac{d(\gamma-1)}{2-\gamma}$, $q=\frac{d}{\gamma'}$ agrees with the one already found in \cite[Theorem 1.1]{GMP14}. Indeed, the authors in \cite{GMP14} proved that $|u|^\tau\in W_0^{1,2}(\Omega)$ with $\tau= \frac{(d-2)(\gamma-1)}{2(2-\gamma)}$ and the same type of dependence in the constant of the estimate. By Sobolev embedding this yields $|u|^\tau\in L^{\frac{2d}{d-2}}$, which, by the definition of $\tau$, yields $u\in L^p$ with $p$ as above.
\end{rem}

\subsection{Maximal $L^q$-regularity for the Dirichlet problem}\label{max}

\begin{proof}[Proof of Theorem \ref{main}]
We first use Theorem \ref{cz} to conclude that the strong solutions to \eqref{hj} satisfy the estimate
\[
\|u\|_{W^{2,q}(\Omega)}\leq C_1(\||Du|\|_{L^{q\gamma}(\Omega)}^\gamma+\|f\|_{L^q(\Omega)}+1)
\]
where $C_1$ depends on $C_H,q,d,\Omega,\lambda$. We then proceed by interpolation as in \cite{CGpar} (see also the references therein for other regularity results obtained through a similar scheme). First, we recall that from Corollary \ref{finalint} we have
\begin{equation}\label{Ls}
\|u\|_{L^s(\Omega)}\leq C
\end{equation}
for any $s\leq p=\frac{dq}{d-2q}$ if $q<\frac{d}{2}$, $s\leq \infty$ when $q>\frac{d}{2}$, with $C$ depending on the previous quantities. The Gagliardo-Nirenberg inequality, cf \cite{Nirenberg}, gives the existence of a positive constant $C_{\mathrm{GN}}$ such that
\begin{equation}\label{Dugammaq}
\||Du|\|_{L^{q\gamma}(\Omega)}\leq C_{\mathrm{GN}}\|D^2u\|_{L^q(\Omega)}^{\theta}\|u\|_{L^{s}(\Omega)}^{1-\theta}+\|u\|_{L^s(\Omega)},
\end{equation}
for $s\in[1,\infty]$ and $\theta\in[1/2,1)$ such that
\[
\frac{1}{\gamma q}=\frac1d+\theta\left(\frac1q-\frac{2}{d}\right)+(1-\theta)\frac1s.
\]
Since $q>\frac{d}{\gamma'}$, it follows that $p>\frac{d(\gamma-1)}{2-\gamma}$, and we are then free to choose $s$ close to $\frac{d(\gamma-1)}{2-\gamma}$ to ensure $\theta\in[1/2,1/\gamma)$. We finally conclude by plugging \eqref{Ls} into \eqref{Dugammaq} and get
\[
\|u\|_{W^{2,q}(\Omega)}\leq C_2(\|D^2u\|_{L^q(\Omega)}^{\theta\gamma}+\|f\|_{L^q(\Omega)}+1)\ ,
\]
which gives the estimate by applying the Young's inequality since $\theta\gamma<1$.

\end{proof}
\begin{proof}[Proof of Theorem \ref{main2}]
We first discuss (i). Let $w=u-u_k$, $k$ to be chosen later, where $u_k$ solves the Dirichlet problem with truncated right-hand side \eqref{hjtrunc}. By Proposition \ref{stabLpest} we have the existence of a constant $C>0$ such that
\[
\|w\|_{L^p(\Omega)}\leq C\|f-T_k(f)\|_{L^q(\Omega)}\ ,p=\frac{d(\gamma-1)}{2-\gamma}\ .
\]
In view of the Gagliardo-Nirenberg inequality we get
\[
\|Dw\|_{L^{q\gamma}(\Omega)}\leq C_2\|w\|_{W^{2,q}(\Omega)}^\frac1\gamma\|w\|_{L^p(\Omega)}^{1-\frac1\gamma}\ ,
\]
where $C_2=C_2(d,p,q,\Omega)$. Thus, we write
\begin{equation}\label{Dww}
\|Dw\|_{L^{q\gamma}(\Omega)}^\gamma\leq C_2^\gamma C^{\gamma-1}\|f-T_k(f)\|_{L^q(\Omega)}^{\gamma-1}\|w\|_{W^{2,q}(\Omega)}\ .
\end{equation}
We further observe that $w$ solves the Dirichlet problem
\begin{equation}\label{auxw}
\begin{cases}
-\Delta w+\lambda w=H(x,Du_k(x))-H(x,Du(x))+f(x)-T_k(f(x))&\text{ in }\Omega\ ,\\
w=0&\text{ on }\partial\Omega\ ,
\end{cases}
\end{equation}
and, in view of the growth assumptions \eqref{H}, we have $|D_pH(x,p)|\leq \widetilde C_H(|p|^{\gamma-1}+1)$ for some $\widetilde C_H>0$, and therefore we deduce by the Young's inequality
\begin{multline*}
|H(x,Du_k(x))-H(x,Du(x))|\leq |Dw(x)|\cdot \max\{|D_pH(x,Du_k(x))|,|D_pH(x,Du(x))|\}\\
\leq C_3(|Du_k(x)|^\gamma+|Du(x)|^\gamma+|Dw(x)|^\gamma+1)\leq C_4(|Du_k(x)|^\gamma+|Dw(x)|^\gamma+1)\ ,
\end{multline*}
where $C_3,C_4$ depend only on $C_H$. We then apply Theorem \ref{cz} to the Dirichlet problem \eqref{auxw} to find
\begin{multline*}
\|w\|_{W^{2,q}(\Omega)}\leq C_5(\|H(x,Du_k(x))-H(x,Du(x))\|_{L^q(\Omega)}+\|f-T_k(f)\|_{L^q(\Omega)})\\
\leq C_6\|Dw\|_{L^{q\gamma}(\Omega)}^\gamma+C_6(\|Du_k\|_{L^q(\Omega)}^\gamma+\|f-T_k(f)\|_{L^q(\Omega)}+1)\ .
\end{multline*}
We plug the latter inequality into \eqref{Dww} to obtain
\begin{multline*}
\|Dw\|_{L^{q\gamma}(\Omega)}^\gamma\leq C_6C_2^\gamma C^{\gamma-1}\|f-T_k(f)\|_{L^q(\Omega)}^{\gamma-1}\|Dw\|_{L^{q\gamma}(\Omega)}^\gamma\\
+C_6(\|Du_k\|_{L^q(\Omega)}^\gamma+\|f-T_k(f)\|_{L^q(\Omega)}+1)C_2^\gamma C^{\gamma-1}\|f-T_k(f)\|_{L^q(\Omega)}^{\gamma-1}\ ,
\end{multline*}
for a constant $C_6=C_6(d,q,C_H)$. We then pick a certain $\bar k$ large enough to have
\begin{equation}\label{ksegnato}
C_6C_2^\gamma C^{\gamma-1}\|f-T_{\bar k}(f)\|_{L^q(\Omega)}^{\gamma-1}\leq\frac12
\end{equation}
and conclude
\[
\|Dw\|_{L^{q\gamma}(\Omega)}^\gamma\leq \|Du_{\bar k}\|_{L^{q\gamma}(\Omega)}^\gamma+\|f-T_{\bar k}(f)\|_{L^q(\Omega)}+1
\leq \|Du_{\bar k}\|_{L^{q\gamma}(\Omega)}^\gamma+2\|f\|_{L^q(\Omega)}+1\ .
\]
We can then apply Theorem \ref{main} to estimate $\|Du_{\bar k}\|_{q\gamma}$ since $T_{\bar k}(f)\in L^\infty$ (actually, $Du$ turns out to be even bounded via the results in \cite{Lions85}) and get for any $\bar q>q$ the bound
\[
\|Du_{\bar k}\|_{L^{\bar q\gamma}(\Omega)}\leq C_{\bar k}\ ,
\]
where $C_{\bar k}$ depends on $\bar k,q,d,C_H,|\Omega|$ (indeed, $\|T_{\bar k}(f)\|_{L^{\bar q}(\Omega)}\leq \bar k|\Omega|^{\frac{1}{\bar q}}$). This allows to conclude the desired estimate, since by the definition of $w$ we have
\[
\|Du\|_{L^{q\gamma}(\Omega)}\leq \|Dw\|_{L^{q\gamma}(\Omega)}+\|Du_{\bar k}\|_{L^{q\gamma}(\Omega)}\ .
\]
The case (ii) can be proved in the same manner and the estimates will depend in addition on $\|u\|_{L^q}$, as observed above the statement of Theorem \ref{main}. A full estimate depending only on $f$ can be achieved as pointed out in Remark \ref{remfull}.
\end{proof}
\begin{rem}
The constant of the estimate in Theorem \ref{main2} remains bounded when $f$ varies in a bounded and equi-integrable set $\mathcal{F}\subset L^q(\Omega)$, see the previous Remark \ref{equi}. Indeed, the constant of the estimate depends on $\bar k$ that appears in \eqref{ksegnato}, where $\bar c=(2C_6C_2^\gamma C^{\gamma-1})^{-1}$ is independent of $f\in\mathcal{F}$.
\end{rem}
\begin{rem}\label{sulorentz}
In view of the recent works for the Hamilton-Jacobi equation in \cite{PhucCPDE,MercaldoCPAA}, one might expect the validity of maximal regularity estimates for the Dirichlet problem when the source term $f$ belongs to the end-point Lorentz scale $L(\frac{d}{\gamma'},\infty)$. Some results in this direction when $\gamma<\frac{d+2}{d}$ without zero-th order terms have been obtained in \cite{MercaldoCPAA}. Still, when $f\in L(d,1)$ one should expect boundedness of the gradient, at least in the context of interior estimates and problems with periodic data, as analyzed in \cite{KM} for different classes of nonlinear elliptic equations. The general case at these critical integrability scales will be the matter of a future research.
\end{rem}
\section{Beyond the subquadratic growth}\label{quad}
As already discussed in \cite{CGpar} (see also \cite{BensoussanBook} for a similar approach), to address the problem of maximal regularity when $\gamma= 2$ in \eqref{H} (and also $\gamma>2$), one can no longer apply the classical Gagliardo-Nirenberg interpolation inequality on Lebesgue spaces with $s=\infty$, since $\theta\gamma=2\theta=1$, which in turn prevents from the absorption of the second order derivative term on the left-hand side of the estimate. Though the purely quadratic case can be investigated using an exponential change of variable, as discussed in \cite{Napoli}, the case of a general Hamiltonian with quadratic growth can be addressed exploiting the Miranda-Nirenberg interpolation inequalities, that allow to interpolate the first order Sobolev space with a H\"older class, cf e.g. \cite[Section 1.1 p. 9]{BensoussanBook}. Indeed, for $\theta\in\left[\frac{1-\alpha}{2-\alpha},1\right]$, $\alpha\in(0,1)$, and $r,q,\theta$ satisfying the compatibility conditions
\[
\frac{1}{2q}=\frac{1}{d}+\theta\left(\frac{1}{r}-\frac2d\right)-(1-\theta)\frac{\alpha}{d},
\]
any function $u\in W^{2,r}(\Omega)\cap C^\alpha(\overline{\Omega})$ belongs to $W^{1,2q}(\Omega)$. Then we have, choosing $\theta=\frac{1-\alpha}{2-\alpha}$, the strict inequality $2\theta<1$ and the estimate
\[
\|Du\|_{L^{2q}(\Omega)}\leq C\|D^2u\|_{L^r(\Omega)}^{1-\theta}[u]_{C^\alpha(\Omega)}^\theta+[u]_{C^\alpha(\Omega)}\ ,r=q\frac{2-2\alpha}{2-\alpha}\ ,
\]
where $[\cdot]_{C^\alpha}$ is the H\"older seminorm.
Then, since $r<q$ for any $\alpha>0$, one can conclude the statement once solutions to the Dirichlet problem \eqref{hj} satisfy a H\"older estimate. Since $u\in L^\infty(\Omega)$ by Corollary \ref{finalint} when $q>\frac{d}{2}$, we can regard \eqref{hj} as
\[
-\Delta u+H(x,Du)=-\lambda u+f=:g\in L^q\ ,q>\frac{d}{2}
\]
and conclude invoking the H\"older estimates for the Dirichlet problem in \cite[Theorem 2.2 p.441]{LuBook} for the case of the quadratic growth. The case of arbitrary first-order terms with natural growth in the endpoint case $q=\frac{d}{2}$ remains an open problem, even in the time-dependent setting under the (parabolic) assumption $q=\frac{d+2}{2}$, and it will be the matter of future research. The case $\gamma>2$ is deeper and has been addressed in \cite{CV} for elliptic problems, in the context of interior estimates. Instead, global bounds for problems equipped with Dirichlet boundary conditions in the superquadratic regime remains at this stage an open problem. Nonetheless, following the proof of \cite{CGpar}, one can get maximal regularity through the adjoint method under the slightly stronger assumption $q>\frac{d(\gamma-1)}{2}$, where $\frac{d(\gamma-1)}{2}>\frac{d(\gamma-1)}{\gamma}$ for $\gamma>2$.

\section{Further remarks on maximal regularity and some open problems}\label{op}
The maximal regularity results proved in this paper, together with the previous ones in \cite{CGell} and the subsequent developments in \cite{CGpar,CV,GP}, provide bounds of the form
\[
\|D^2u\|_q+\||Du|^\gamma\|_q\leq C_1
\]
where $C_1$ is an implicit constant depending on the data of the problem, and in particular on $\|f\|_q$. Following P.-L. Lions \cite{Napoli}, we can regard this statement as a weak version of the maximal regularity. One may ask whether a stronger statement can be true, i.e. regarding the possible validity of an estimate of the form
\[
\|D^2u\|_q+\||Du|^\gamma\|_q\leq C_2\|f-\lambda u\|_q,
\]
where the maximal regularity estimate holds with an explicit linear (or even nonlinear) dependence respect to the right-hand side datum $f$. In this case we refer to this statement as the strong maximal $L^q$-regularity. When $q$ is close to $d$ such a strong estimate was proved in \cite{Lions85}, and later extended in \cite{CGL} for problems driven by the $p$-Laplacian and power-growth gradient terms. A strong form of the maximal regularity for the $p$-Poisson equation without lower order terms can be found in \cite{CianchiMazya}.\\
 When $\frac{d}{\gamma'}\leq q<d$ the validity of this stronger form of the maximal regularity for equation \eqref{hj} is, at this stage, an open problem in its full generality. Nonetheless, some cases as $d=1$ or $d=q=\gamma=2$, and also $q>d-c(d,\gamma)(>\frac{d}{\gamma'})$ when $d\geq3$, in the viscous case have been already discussed and proved by P.-L. Lions \cite{Napoli}.

\appendix

\section{Auxiliary results}\label{app}
We recall the following $W^{2,q}$ a priori estimate for strong solutions to linear elliptic problems.

\begin{thm}\label{cz}
Let $\Omega\subset \R^d$ be a $C^2$ bounded domain. If $u\in W^{2,q}(\Omega)\cap W^{1,q}_0(\Omega)$, $g\in L^q(\Omega)$ satisfies $-\Delta u+\lambda u=g$ a.e. in $\Omega$, $\lambda\in\R$, $\lambda>0$, then 
\[
\|u\|_{W^{2,q}(\Omega)}\leq C_1\|g\|_{L^q(\Omega)}
\]
where $C_1$ depends only on $d,q, \lambda, \Omega$. As a consequence, any strong solution $u\in W^{2,q}(\Omega)\cap W^{1,q}_0(\Omega)$ to \eqref{hj} satisfies
\[
\|u\|_{W^{2,q}(\Omega)}\leq C_2(\||Du|\|_{L^{q\gamma}(\Omega)}^\gamma+\|f\|_{L^q(\Omega)}+1)
\]
where $C_2$ depends on $d,q,\lambda, \Omega,C_H$.
\end{thm}
\begin{proof}
The first statement is a classical Calder\'on-Zygmund estimate that can be found in \cite[Lemma 9.17]{GT} and \cite[Theorem 6.3]{ChenWu}. The second result follows readily from the first by observing that if $u$ is a strong solution to \eqref{hj}, then it solves 
\[
-\Delta u+\lambda u=-H(x,Du)+g(x)
\]
with the same boundary conditions. Then, the estimate is a consequence of the assumptions \eqref{H} and the properties of Lebesgue spaces.
\end{proof}


\small

\begin{thebibliography}{10}

\bibitem{AF}
R.~A. Adams and J.~J.~F. Fournier.
\newblock {\em Sobolev spaces}, volume 140 of {\em Pure and Applied Mathematics
  (Amsterdam)}.
\newblock Elsevier/Academic Press, Amsterdam, second edition, 2003.

\bibitem{Agmon}
S.~Agmon.
\newblock The {$L_{p}$} approach to the {D}irichlet problem. {I}. {R}egularity
  theorems.
\newblock {\em Ann. Scuola Norm. Sup. Pisa Cl. Sci. (3)}, 13:405--448, 1959.

\bibitem{AFM15}
A.~Alvino, V.~Ferone, and A.~Mercaldo.
\newblock Sharp a priori estimates for a class of nonlinear elliptic equations
  with lower order terms.
\newblock {\em Ann. Mat. Pura Appl. (4)}, 194(4):1169--1201, 2015.

\bibitem{AC}
H.~Amann and M.~G. Crandall.
\newblock On some existence theorems for semi-linear elliptic equations.
\newblock {\em Indiana Univ. Math. J.}, 27(5):779--790, 1978.

\bibitem{BardiPerthame}
M.~Bardi and B.~Perthame.
\newblock Uniform estimates for some degenerating quasilinear elliptic
  equations and a bound on the {H}arnack constant for linear equations.
\newblock {\em Asymptotic Anal.}, 4(1):1--16, 1991.

\bibitem{BensFrehse}
A.~Bensoussan and J.~Frehse.
\newblock On {B}ellman systems without zero order term in the context of risk
  sensitive differential games.
\newblock In {\em Proceedings of {P}artial {D}ifferential {E}quations and
  {A}pplications ({O}lomouc, 1999)}, volume 126, pages 275--280, 2001.

\bibitem{BensoussanBook}
A.~Bensoussan and J.~Frehse.
\newblock {\em Regularity results for nonlinear elliptic systems and
  applications}, volume 151 of {\em Applied Mathematical Sciences}.
\newblock Springer-Verlag, Berlin, 2002.

\bibitem{MercaldoCPAA}
M.~F. Betta, R.~Di~Nardo, A.~Mercaldo, and A.~Perrotta.
\newblock Gradient estimates and comparison principle for some nonlinear
  elliptic equations.
\newblock {\em Commun. Pure Appl. Anal.}, 14(3):897--922, 2015.

\bibitem{BocUMI}
L.~Boccardo.
\newblock Some developments on {D}irichlet problems with discontinuous
  coefficients.
\newblock {\em Boll. Unione Mat. Ital. (9)}, 2(1):285--297, 2009.

\bibitem{BocJDE}
L.~Boccardo.
\newblock Dirichlet problems with singular convection terms and applications.
\newblock {\em J. Differential Equations}, 258(7):2290--2314, 2015.
\bibitem{BarlesPorretta}
G.~Barles and A.~Porretta.
\newblock Uniqueness for unbounded solutions to stationary viscous
  {H}amilton-{J}acobi equations.
\newblock {\em Ann. Sc. Norm. Super. Pisa Cl. Sci. (5)}, 5(1):107--136, 2006.

\bibitem{BGae}
L.~Boccardo and T.~Gallou\"{e}t.
\newblock Nonlinear elliptic equations with right-hand side measures.
\newblock {\em Comm. Partial Differential Equations}, 17(3-4):641--655, 1992.

\bibitem{BO}
L.~Boccardo and L.~Orsina.
\newblock The duality method for mean field games systems.
\newblock {\em To appear in Comm. Pure Appl. Anal.}, 2022.
\newblock { DOI:10.3934/cpaa.2022021}.

\bibitem{BocBuc}
L.~Boccardo, S.~Buccheri, and G.~R. Cirmi.
\newblock Two linear noncoercive {D}irichlet problems in duality.
\newblock {\em Milan J. Math.}, 86(1):97--104, 2018.

\bibitem{BMP92}
L.~Boccardo, F.~Murat, and J.-P. Puel.
\newblock {$L^\infty$} estimate for some nonlinear elliptic partial
  differential equations and application to an existence result.
\newblock {\em SIAM J. Math. Anal.}, 23(2):326--333, 1992.

\bibitem{BOPpar}
L.~Boccardo, L.~Orsina, and A.~Porretta.
\newblock Some noncoercive parabolic equations with lower order terms in
  divergence form.
\newblock {\em J. Evol. Equ.}, 3(3):407--418, 2003.

\bibitem{BOP}
L.~Boccardo, L.~Orsina, and A.~Porretta.
\newblock Strongly coupled elliptic equations related to mean-field games
  systems.
\newblock {\em J. Differential Equations}, 261(3):1796--1834, 2016.

\bibitem{BKR}
V.~I. Bogachev, N.~V. Krylov, and M.~R\"{o}ckner.
\newblock Elliptic regularity and essential self-adjointness of {D}irichlet
  operators on {$\bold R^{n}$}.
\newblock {\em Ann. Scuola Norm. Sup. Pisa Cl. Sci. (4)}, 24(3):451--461, 1997.

\bibitem{BKRS}
V.~I. Bogachev, N.~V. Krylov, M.~R\"ockner, and S.~V. Shaposhnikov.
\newblock {\em Fokker-{P}lanck-{K}olmogorov equations}, volume 207 of {\em
  Mathematical Surveys and Monographs}.
\newblock American Mathematical Society, Providence, RI, 2015.

\bibitem{BrezisAMO}
H.~Brezis.
\newblock Semilinear equations in {${\bf R}^N$} without condition at infinity.
\newblock {\em Appl. Math. Optim.}, 12(3):271--282, 1984.


\bibitem{CC}
A.~Cesaroni and M.~Cirant.
\newblock Introduction to variational methods for viscous ergodic mean-field
  games with local coupling.
\newblock In {\em Contemporary research in elliptic {PDE}s and related topics},
  volume~33 of {\em Springer INdAM Ser.}, pages 221--246. Springer, Cham, 2019.
\bibitem{CianchiMazya}
A.~Cianchi and V.~Maz'ya.
\newblock Gradient regularity via rearrangements for {$p$}-{L}aplacian type
  elliptic boundary value problems.
\newblock {\em J. Eur. Math. Soc. (JEMS)}, 16(3):571--595, 2014.
\bibitem{ChenWu}
Y.-Z. Chen and L.-C. Wu.
\newblock {\em Second order elliptic equations and elliptic systems}, volume
  174 of {\em Translations of Mathematical Monographs}.
\newblock American Mathematical Society, Providence, RI, 1998.

\bibitem{CCPDE}
M.~Cirant.
\newblock Stationary focusing mean-field games.
\newblock {\em Comm. Partial Differential Equations}, 41(8):1324--1346, 2016.

\bibitem{c22par}
M.~Cirant.
\newblock On the improvement of {H}\"older seminorms in superquadratic
  {H}amilton-{J}acobi equations.
\newblock arXiv:2208.00082, 2022.

\bibitem{cg20}
M.~Cirant and A.~Goffi.
\newblock Lipschitz regularity for viscous {H}amilton-{J}acobi equations with
  {$L^p$} terms.
\newblock {\em Ann. Inst. H. Poincar\'{e} Anal. Non Lin\'{e}aire},
  37(4):757--784, 2020.

\bibitem{CGpar}
M.~Cirant and A.~Goffi.
\newblock Maximal {$L^q$}-regularity for parabolic {H}amilton-{J}acobi
  equations and applications to {M}ean {F}ield {G}ames.
\newblock {\em Ann. PDE}, 7(2):Paper No. 19, 40, 2021.

\bibitem{CGell}
M.~Cirant and A.~Goffi.
\newblock On the problem of maximal {$L^q$}-regularity for viscous
  {H}amilton-{J}acobi equations.
\newblock {\em Arch. Ration. Mech. Anal.}, 240(3):1521--1534, 2021.

\bibitem{CV}
M.~Cirant and G.~Verzini.
\newblock Local {H}\"older and maximal regularity of solutions of elliptic
  equations with superquadratic gradient terms.
\newblock {\em Adv. Math.}, 409:108700, 16, 2022.

\bibitem{CGL}
M.~Cirant, A.~Goffi and T.~Leonori.
\newblock Gradient estimates for quasilinear elliptic Neumann problems with unbounded first-order terms, arXiv: 2211.03760, 2022.

\bibitem{DAP}
A.~Dall'Aglio and A.~Porretta.
\newblock Local and global regularity of weak solutions of elliptic equations
  with superquadratic {H}amiltonian.
\newblock {\em Trans. Amer. Math. Soc.}, 367(5):3017--3039, 2015.

\bibitem{EvansTAMS}
L.~C. Evans.
\newblock Some estimates for nondivergence structure, second order elliptic
  equations.
\newblock {\em Trans. Amer. Math. Soc.}, 287(2):701--712, 1985.

\bibitem{Evansacc}
L.~C. Evans.
\newblock Adjoint and compensated compactness methods for {H}amilton-{J}acobi
  {PDE}.
\newblock {\em Arch. Ration. Mech. Anal.}, 197(3):1053--1088, 2010.

\bibitem{EvansSmart}
L.~C. Evans and C.~K. Smart.
\newblock Adjoint methods for the infinity {L}aplacian partial differential
  equation.
\newblock {\em Arch. Ration. Mech. Anal.}, 201(1):87--113, 2011.

\bibitem{FM}
V.~Ferone and B.~Messano.
\newblock Comparison and existence results for classes of nonlinear
              elliptic equations with general growth in the gradient.
\newblock {\em Adv. Nonlinear Stud.}, 7(1):31--46, 2007.

\bibitem{FL}
I.~Fonseca and G.~Leoni.
\newblock {\em Modern methods in the calculus of variations: {$L^p$} spaces}.
\newblock Springer Monographs in Mathematics. Springer, New York, 2007.



\bibitem{GT}
D.~Gilbarg and N.~S. Trudinger.
\newblock {\em Elliptic partial differential equations of second order}.
\newblock Classics in Mathematics. Springer-Verlag, Berlin, 2001.
\newblock Reprint of the 1998 edition.

\bibitem{GP}
A.~Goffi and F.~Pediconi.
\newblock Sobolev regularity for nonlinear {P}oisson equations with {N}eumann
  boundary conditions on {R}iemannian manifolds.
\newblock Forum Math., 2023, DOI: 10.1515/forum-2022-0119.

\bibitem{Gomesbook}
D.~A. Gomes, E.~A. Pimentel, and V.~Voskanyan.
\newblock {\em Regularity theory for mean-field game systems}.
\newblock SpringerBriefs in Mathematics. Springer, [Cham], 2016.

\bibitem{GMP14}
N.~Grenon, F.~Murat, and A.~Porretta.
\newblock A priori estimates and existence for elliptic equations with gradient
  dependent terms.
\newblock {\em Ann. Sc. Norm. Super. Pisa Cl. Sci. (5)}, 13(1):137--205, 2014.

\bibitem{HMV99}
K.~Hansson, V.~G. Maz'ya, and I.~E. Verbitsky.
\newblock Criteria of solvability for multidimensional {R}iccati equations.
\newblock {\em Ark. Mat.}, 37(1):87--120, 1999.

\bibitem{KK}
J.~L. Kazdan and R.~J. Kramer.
\newblock Invariant criteria for existence of solutions to second-order
  quasilinear elliptic equations.
\newblock {\em Comm. Pure Appl. Math.}, 31(5):619--645, 1978.

\bibitem{KM}
T.~Kuusi and G.~Mingione.
\newblock Guide to nonlinear potential estimates.
\newblock {\em Bull. Math. Sci.}, 4(1):1--82, 2014.

\bibitem{KS}
S.~Koike and A.~\'{S}wiech.
\newblock Maximum principle for fully nonlinear equations via the iterated
  comparison function method.
\newblock {\em Math. Ann.}, 339(2):461--484, 2007.

\bibitem{LuBook}
O.~A. Ladyzhenskaya and N.~N. Ural'tseva.
\newblock {\em Linear and quasilinear elliptic equations}.
\newblock Academic Press, New York-London, 1968.

\bibitem{LL07}
J.-M. Lasry and P.-L. Lions.
\newblock Mean field games.
\newblock {\em Jpn. J. Math.}, 2(1):229--260, 2007.

\bibitem{Lin}
F.-H. Lin.
\newblock Second derivative {$L^p$}-estimates for elliptic equations of
  nondivergent type.
\newblock {\em Proc. Amer. Math. Soc.}, 96(3):447--451, 1986.

\bibitem{Lions85}
P.-L. Lions.
\newblock Quelques remarques sur les probl\`emes elliptiques quasilin\'{e}aires
  du second ordre.
\newblock {\em J. Analyse Math.}, 45:234--254, 1985.

\bibitem{LionsSeminar}
P.-L. Lions.
\newblock Recorded video of S\'eminaire de Math\'ematiques appliqu\'ees at
  Coll\'ege de France, available at
  \url{https://www.college-de-france.fr/site/pierre-louis-lions/seminar-2014-11-14-11h15.htm},
  November 14, 2014.

\bibitem{Napoli}
P.-L. Lions.
\newblock {O}n {M}ean {F}ield {G}ames.
\newblock Seminar at the conference ``{T}opics in {E}lliptic and {P}arabolic
  {PDE}s", Napoli, September 11-12, 2014.
  
\bibitem{MaSoBook}
A.~Maugeri, D.~K. Palagachev, and L.~G. Softova.
\newblock {\em Elliptic and parabolic equations with discontinuous
  coefficients}, volume 109 of {\em Mathematical Research}.
\newblock Wiley-VCH Verlag Berlin GmbH, Berlin, 2000.

\bibitem{MPRell}
G.~Metafune, D.~Pallara, and A.~Rhandi.
\newblock Global properties of invariant measures.
\newblock {\em J. Funct. Anal.}, 223(2):396--424, 2005.

\bibitem{Nirenberg}
L.~Nirenberg.
\newblock On elliptic partial differential equations.
\newblock {\em Ann. Scuola Norm. Sup. Pisa Cl. Sci. (3)}, 13:115--162, 1959.

\bibitem{PhucCPDE}
N.~C. Phuc.
\newblock Quasilinear {R}iccati type equations with super-critical exponents.
\newblock {\em Comm. Partial Differential Equations}, 35(11):1958--1981, 2010.

\bibitem{Pierre}
M.~Pierre.
\newblock Global existence in reaction-diffusion systems with control of mass:
  a survey.
\newblock {\em Milan J. Math.}, 78(2):417--455, 2010.

\bibitem{NotePorr}
A.~Porretta.
\newblock {E}lliptic equations with first order terms.
\newblock Notes of the CIMPA school, Alexandria, 2009.

\bibitem{PorroLincei}
A.~Porretta.
\newblock The ``ergodic limit'' for a viscous {H}amilton-{J}acobi equation with
  {D}irichlet conditions.
\newblock {\em Atti Accad. Naz. Lincei Rend. Lincei Mat. Appl.}, 21(1):59--78,
  2010.

\bibitem{S}
M.~Schechter.
\newblock On {$L^{p}$} estimates and regularity. {I}.
\newblock {\em Amer. J. Math.}, 85:1--13, 1963.

\bibitem{TranBook}
H.~V. Tran.
\newblock {\em Hamilton-{J}acobi equations}, volume 213 of {\em Graduate
  Studies in Mathematics}.
\newblock American Mathematical Society, Providence, RI, 2021.
\newblock Theory and applications.

\end{thebibliography}
\end{document}